\documentclass{amsart}

\usepackage{amsmath, amssymb,amsfonts, amscd, xr}
\usepackage{glossaries}
\usepackage{imakeidx}
\usepackage[shortlabels]{enumitem}
\usepackage{hyperref}
\input{olddefs.tex}

\numberwithin{claim}{subsection}

\def\TO#1{^{(#1)}}

\newcounter{my_enumerate_counter}

\title{Biba's trick, with applications}

\author{Ilijas Farah}
\address{Department of Mathematics and Statistics\\
	York University\\
	4700 Keele Street\\
	North York, Ontario\\ Canada, M3J 1P3\\
	and 
	Ma\-te\-ma\-ti\-\v cki Institut SANU\\
	Kneza Mihaila 36\\
	11\,000 Beograd, p.p. 367\\
	Serbia}
\email{ifarah@yorku.ca}
\urladdr{https://ifarah.mathstats.yorku.ca}
\thanks{Partially supported by NSERC}
\date{\today}

\begin{document}
\maketitle

\begin{abstract}
	We give another bit of evidence that forcing axioms provide proper framework for rigidity of quotient structures, by improving the OCA lifting theorem proved by the author in late 20th century and greatly simplifying its proof. In the assumptions of this theorem, Martin’s Axiom is weakened to $\MAsigma$ and, in case of countably generated ideals, completely removed. (In case of the Frech\' et ideal this was done by De Bondt, Vignati, and the author.)  We also extend the conclusion of author’s 2004 lifting theorem from a lifting result for countably 3204-determined ideals to one for countably 80-determined ideals and  also weaken $\PFA$ to $\OCAT$ and $\MAsigma$ in its assumptions. 
\end{abstract}

This paper presents the latest progress in the corona rigidity program as surveyed in \cite{farah2022corona}. It is about the oldest (and most studied) thread in this program, asking when an isomorphism between quotient Boolean algebras of the form $\cPN/\cI$ for an analytic ideal $\cI$ on $\bbN$ has a completely additive lifting.\footnote{In this paper every lifting is a set-theoretic lifting, with no algebraic properties assumed unless explicitly stated.}  
Here $\cPN$ is considered with the Cantor set topology and `analytic' means `continuous image of a Polish space', while a `completely additive lifting' is one of the form 
\[
X\mapsto h^{-1}(X)
\]
for a function $h\colon \bbN\to \bbN$.  
Such a lifting theorem is typically proven in two stages. The first stage uses forcing axioms to produce a continuous lifting. The second stage uses Ulam-stability to produce a completely additive lifting from a continuous one. This has been done for a large class of analytic ideals $\cI$, including all ideals mentioned in this paper.  
 The second stage deals with absolute statements and belongs to a different story  (see \cite[\S 2 and \S 5.3]{farah2022corona}).

For ideals $\NWD( \bbQ)$, $\NULL( \bbQ)$,
and $\cZ_W$ see Lemma~\ref{L.SolMa}. 
Pathological ideals exist (this follows from \cite[Theorem~1.8.6]{Fa:AQ}) but they will play no role here.  

\begin{thm} \label{T.Main} Assume  $\OCAT$  and $\MAsigma$. Suppose that $\cI$ is a nonpathological  $F_\sigma$ ideal, a nonpathological analytic P-ideal, or any of the ideals $\NWD( \bbQ)$, $\NULL( \bbQ)$,
	or $\cZ_W$. 	If $\cI'$ is an analytic ideal and  $\Phi\colon \cP(\bbN)/\cI'\to \cP(\bbN)/\cI$ is an isomorphism, then it has a completely additive lifting. In particular, $\cI$ and $\cI'$ are Rudin--Keisler isomorphic. 
	\end{thm}

Theorem~\ref{T.Main} is a consequence of the main technical result of this paper, Theorem~\ref{T.OCA-l}. 
This result strengthens the OCA Lifting Theorem of \cite{Fa:AQ}.  The proof given here replaces an excruciating thirty pages of \cite[\S 3.8--\S 3.13]{Fa:AQ}.   It also strengthens \cite{Fa:Luzin} where PFA had been used to prove a weaker result. 
The results presented here are proven by combining the methods of \cite{Fa:AQ} and \cite{Fa:Luzin} with a refinement of the techniques introduced in \cite{de2023trivial}. 
The question whether forcing axioms (even MM, \cite{FoMaShe:Martin}) can be used to extend the conclusion of Theorem~\ref{T.OCA-l} to ideals that are not countably determined by closed approximations (see \S\ref{S.CountablyDetermined}) or even not $F_{\sigma\delta}$ is as open as ever  (see \cite[\S 7.2]{farah2022corona}).

The following is a consequence of Theorem~\ref{T.OCAsharp-Fin}, proven at the end of \S\ref{S.UniformizationFin}, 

\begin{thm}\label{T.OCAonly}
	Assume $\OCAT$. If $\cI$ is a countably generated ideal on $\bbN$ then all automorphisms of $\cPN/\cI$ are trivial (i.e., given by a bijection between sets whose complements belong to $\cI$). 
\end{thm}

Consistency of the assertion that all automorphisms of $\cPN/\Fin$ are trivial was proven in \cite{Sh:Proper}. It was show to follow from PFA in \cite{ShSte:PFA}, from $\OCAT$ and $\MA$ in \cite{Ve:OCA}, and from $\OCAT$ alone in  \cite[Theorem 1]{de2023trivial}. For the other countably generated ideal on $\bbN$, $\FinO$, the conclusion of Theorem~\ref{T.OCAonly} was proven from $\OCAT$ and $\MA$ in \cite[Theorem 1.9.2]{Fa:AQ}.

For current state of the art in the quotient rigidity  program see~\cite{farah2022corona}. Results such as Theorem~\ref{T.Main} may give evidence that forcing axioms provide a preferable ambient theory (as in \cite{magidor201some}) for study of quotient structures ranging (at least) from Boolean algebras to \cstar-algebras.

\subsection*{Apologies and acknowledgments} 
 I am grateful to Biba the~47\% miniature Schnauzer for taking me to Toronto's Riverdale  Park  West on October 6, 2024.  The germ of the idea that resulted in the proof of \cite[Proposition~5.6]{de2023trivial} and the simplification of the proof of the OCA lifting theorem of \cite{Fa:AQ} that is at the core of this paper occurred  during this pleasant walk. 
I would also like to thank Alan Dow and Ben De Bondt for pointing out to a few typos in the original version of this paper, and to the anonymous referee for a very helpful report. 
\section{Preliminaries}

In this section we give preliminaries, partly in order to introduce concepts and notation and partly because some of the results are not available in the literature in the form appropriate for us. 
We will write $A\Subset B$ if $A\subseteq B$ and $A$ is finite. 

\subsection{Approximations to ideals and ldeals countably determined by closed approximations}
\label{S.CountablyDetermined}

For subsets $\cA$ and $\cB$ of $\cPN$ consider their `Minkowski union', 
\[
\cA\ucup \cB=\{A\cup B\mid A\in \cA,B\in \cB \}. 
\]
The notion of a closed approximation to an ideal first appeared in \cite{Just:WAT}. 

\begin{definition}\label{Def.hereditary}	
	A set $\cK\subseteq \p(\N)$ is \emph{hereditary} if $A\subseteq B\in \cK$ implies $A\in \cK$.
	If $\cK$ and $ \calL$ are families of subsets of $\N$, then
	$$
	\cK\ucup  \calL=\{K\cup L: K\in \cK\text{ and } L\in  \calL\}
	$$
	and
	$$
	\cK^k=\cK\ucup \dots \ucup \cK=\{A_1\cup A_2\cup \dots \cup A_k: A_i\in \cK\text{ for $i\leq k$}\}.
	$$
	Note that $\cK\ucup \Fin=\bigcup_n \{A\mid A\setminus n\in \cK\}$. We say that $\cK\subseteq \cP(\bbN)$ is an \emph{approximation to $\cI$} if it is hereditary  and $\cI\subseteq \cK\ucup \Fin$. An approximation is closed if it is closed (as a subset of the Cantor space).  
\end{definition}

We will be primarily interested in closed approximations. 
The following definition comes from \cite{Fa:Luzin}.

\begin{definition}  \label{D.ctbly.det}
	An ideal $\cI$ is \emph{countably determined} by closed  approximations
	if there are closed   hereditary
	sets $\cK_n$ ($n\in \N$) such that
	\begin{enumerate}
		\item [(a)] $
		\I=\bigcap_{n=1}^\infty (\cK_n\ucup \Fin).
		$
	\end{enumerate}
	If in addition for some  $d\in \N$  we have
	\begin{enumerate}
		\item [(b)] $
		\I=\bigcap_{n=1}^\infty (\cK_n^d\ucup \Fin)
		$
	\end{enumerate}
	we say that $\cI$ is \emph{countably $d$-determined} by closed (analytic, etc.)  approximations.
	If there are closed hereditary $\cK_n$ such that {\rm(b)} holds for all $d\in \N$,
	we say that $\cI$ is \emph{strongly countably determined}
	by $\cK_n$ ($n\in \N$).
\end{definition}

	Clearly, every ideal countably determined by closed approximations
is $F_{\sigma\delta}$.
Our Theorem~\ref{T.OCA-l}, will be proven
for ideals that are countably 80-determined by closed approximations.
No care is taken to assure the optimality of this constant, in particular because all known $F_{\sigma\delta}$ ideals are countably 
$d$-determined by closed approximations for all $d\in \bbN$.  

\begin{lemma} \label{L.StronglyCountablyDetermined}All of the following ideals are strongly countably determined by closed approximations. 
	\label{L.SolMa}
	\begin{enumerate}
		\item  \label{1.L.Strongly} Every $F_\sigma$ ideal.
		\item\label{2.L.Strongly}  Every analytic P-ideal. 
		\item \label{3.L.Strongly} The ideals 
		\begin{align*}
	\NWD( \bbQ)&=\{A\subseteq  \bbQ\cap [0,1]: A\text{ is nowhere dense}\},\\
	\NULL( \bbQ)&=\{A\subseteq  \bbQ\cap [0,1]: \overline A\text{ has Lebesgue measure 0}\}, \\
	\cZ_W&=\biggl\{A: \limsup_n \sup_k \frac{|A\cap [k,k+n)|}n=0\biggr\}.
\end{align*}		
	\end{enumerate}
\end{lemma}

\begin{proof}
	\eqref{1.L.Strongly} 	By \cite[Lemma 6.3]{Laf:Forcing} or 
	\cite[Lemma 1.2(c)]{Maz:F-sigma}, there is a closed hereditary set $\cK$ such that $\cI=\cK\ucup \Fin$.  Since $\cI$ is an ideal, we have $\cK^m\subseteq \cI$ for all $m$, and therefore $\cI=\cK^m\ucup \Fin$ for all $m$. We can let $\cK_n=\cK$ for all $n$.

	\eqref{2.L.Strongly} \cite[Theorem~3.1]{Sol:AnalyticII} implies that 
	$\cI$ is an analytic P-ideal if  and only if it is equal to $\{A\subseteq \bbN\mid \lim_n \varphi(A\setminus n)=0\}$ for
	some lower semicontinuous submeasure $\varphi$ on $\bbN$. Let $\cK_n=\{A\mid  \varphi(A)\leq 2^{-n}\}$. Then $\cI=\bigcap_n (\cK_n\ucup \Fin)$ and $\cK_{n+1}^2=\cK_{n}$ for all $n$. By induction, $\cI=\bigcap_n (\cK_n^m\ucup \Fin)$ for all $m$. 

\eqref{3.L.Strongly} This was proved in \cite[Lemma~2.6]{Fa:Luzin}. 
\end{proof}

Conjecturally, all $F_{\sigma\delta}$ ideals are strongly countably determined by closed approximations. Theorem~\ref{T.OCA-l} below is our main result. The fact that we do not assume that $\ker(\Phi)\supseteq \Fin$ leads to some technical difficulties, but this is well worth it. For example, the ability to describe all copies of $\cP(\bbN)$ inside the quotient $\cP(\bbN)/\Fin$ was used to prove that $\OCAT$ implies that the Lebesgue measure algebra does not embed into $\cPN/\Fin$ (\cite{DoHa:Lebesgue}).

\begin{theorem} 
	\label{T.OCA-l} Assume  $\OCAT$  and $\MAsigma$ and that 
	$\cI$ is a countably 80-determined ideal.  Then for every homomorphism 
	$\Phi\colon \p(\N)\to \p(\N)/\cI$
	there is a continuous function $F\colon \cPN\to \cPN$ that lifts $\Phi$ on a nonmeager  ideal. 
\end{theorem} 

This is an immediate consequence of Proposition~\ref{P.general.Jcont.nonmeager} and Proposition~\ref{P.uniformization}  (for proof see \S\ref{S.Proofs}). 
The analog of Theorem~\ref{T.OCA-l} for ideals that are countably 3204-determined was proved in \cite[Theorem~10.4]{Fa:Luzin} from the Proper Forcing Axiom (PFA). It is not known whether an ideal is 3204-determined if and only if it is 80-determined (or even whether an ideal is strongly countably determined if and only if it is $F_{\sigma\delta}$, see \cite{he2022borel}), but regardless of the status of this problem, since PFA does not follow from $\OCAT$ and $\MAsigma$, Theorem~\ref{T.OCA-l} is a strengthening of  \cite[Theorem~10.4]{Fa:Luzin}. 

\subsection{Nonmeager hereditary sets}
\label{chJNT} 
\begin{definition}
	If  $\X$ is a  family of subsets of $\N$ 
	by $\SI{\hat \X}$ we denote the {\it \II{downwards closure}\/} (or, the {\em
		hereditary closure\/}) of~$\X$, i.e.,
	$\hat \X=\bigcup_{a\in \X}\p(a)$, and we 
	say that $\X$ is \emph{hereditary} if 
	$\X=\hat \X$. 
\end{definition}

The following is \cite[Theorem~3.10.1]{Fa:AQ}. It  is a slight strengthening of results of Jalali--Naini (\cite{JN}) and Talagrand~(\cite{Ta:Compacts}), where analogous result was proved for nonmeager ideals.

\begin{theorem} \label{T.HNM}
	\label{chOCAc.1}
	\begin{enumerate}
		\item 
		If $\X$ is a hereditary subset of $\p(\N)$, 
		then it 
		is nonmeager
		if and only if for every sequence $s_i$ of disjoint finite sets of
		integers there is an infinite $a\subseteq \N$ such that 
		$\bigcup_{i\in a}s_i\in \X$.
		\item   The family of nonmeager hereditary subsets of $\p(\N)$ is closed
		under taking finite intersections of its elements. 
		\item  \label{3.HNM}
		The family of nonmeager hereditary subsets 
		of $\p(\N)$ which are closed under finite changes of their  elements
		is closed under taking countable intersections. \qed 
	\end{enumerate}
\end{theorem}

The following refinement of Theorem~\ref{T.HNM} \eqref{3.HNM} will be useful. 

\begin{lemma}\label{L.HNM.1} If $\calH$  is a hereditary nonmeager set, then there exists $k$ such that for every $s\Subset [k,\infty)$ the set 
	\[
	\calH[s]=\{a\subseteq \bbN\mid s\cup a\in \calH\}
	\]
	is hereditary and nonmeager. 
	In particular, the intersection of $\calH$ with every nonempty open subset of $\cP([k,\infty))$ is relatively nonmeager, and the set $\calH\ucup \cP(k)$ is hereditary, nonmeager, and closed under finite changes of its elements. 
\end{lemma}	

\begin{proof} The set $\calH\ucup \cP(k)$  is clearly hereditary for every $k$. Assume it is meager for all $k$. By Theorem~\ref{T.HNM}, we can recursively find an increasing sequence $k(j)$ and $s(j)\subseteq [k(j), k(j+1))$ such that $\calH[s(j)]$ is meager for all $j$.
	
	Thus $\bigcup_j \calH[s(j)]$ is a meager hereditary set, and by Theorem~\ref{T.HNM} there is a sequence of disjoint finite subsets of $\bbN$, $t(j)$, for $j\in \bbN$, such that the set
	\[
	\{a\subseteq \bbN\mid (\exists^\infty j) t(j)\subseteq a\}
	\]
	 is disjoint from $\bigcup_j \calH[s(j)]$. Recursively find increasing sequences $m(j), n(j)$ such that $u(j)=s(m(j))\cup t(n(j))$ are pairwise disjoint. 
	
	By Theorem~\ref{T.HNM} there is an infinite $X\subseteq \bbN$ such that $a=\bigcup_{j\in X} u(j)\in \calH$. Fix $j\in X$. Then $a\in \calH[s(j)]$ and also $a\notin\bigcup_i \calH[s(i)]$; contradiction. 
	
	For the second claim, it suffices to prove that the intersection of $\calH$ with every basic open subset of $\cP([k,\infty))$ is relatively nonmeager. Since $\calH$ is hereditary, this is equivalent to $\calH[s]$ being nonmeager for every $s\Subset [k,\infty)$. 
\end{proof}

The following will be useful in analysis of liftings that are Baire-measurable on large hereditary sets in \S\ref{S.Coherence}. 

\begin{coro}\label{C.HereditaryComeager}
	If $\cX\subseteq \cP(\bbN)$ is comeager and $\calH$ is a nonmeager hereditary subset of $\cP(\bbN)$, then there are a partition $\bbN=A_0\sqcup A_1$ and sets $C_0\subseteq A_0$ and $C_1\subseteq A_1$ such that for every $X\in \calH$ both $(X\cap A_0)\cup C_1$ and $(X\cap A_1)\cup C_0$ belong to $\calH^2\cap \cX$. 
\end{coro}

\begin{proof} Since $\cX$ is comeager, there are 
	$s(n)\subseteq I(n)\Subset \bbN$  for $n\in \bbN$ such that  $I(n)$ are disjoint and $\cX\supseteq \{X\mid (\exists ^\infty n) X\cap I(n)=s(n)\}$. By Theorem~\ref{T.HNM} there is an infinite $d\subseteq \bbN$ such that $\bigcup_{n\in d} I(n)\in \calH$. Partition $d$ into two infinite sets, $d=d_0\cup d_1$.  Let
	\begin{align*}\textstyle
		A_0=\bigcup_{n\in d_0} I(n),\quad
		A_1=\bbN\setminus A_0,\quad 
		C_0=\bigcup_{n\in d_0} s(n),\quad
		C_1=\bigcup_{n\in d_1} s(n). 
	\end{align*}
	If $X\in \calH$ then, both $Y_0=(X\cap A_0)\cup C_1$ and $Y_1=(X\cap A_1)\cup C_0$ belong to $\calH^2$. Also, $Y_j\in \bigcap_{n\in d_j} [I(n), s(n)]$ for $j=0,1$, hence both $Y_0$ and $Y_1$ belong to $\cX$, as required. 
\end{proof}


\subsection{Tree-like almost disjoint families}
Infinite sets $A$ and $B$ of integers are {\it almost
	disjoint\/} if their intersection is finite, and $A$ {\it almost
	includes\/}  $B$, 
if $B\setminus A$ is finite. 
For $s\in \twolo$ let $|s|$ denote its length (equivalently, its domain).
For $s$ and $t$ in $\twolo$ we write $s\sqsubseteq t$ if $s$ is an initial segment of $t$. If $f$ and $g$ are in $\twoo$ and distinct then $f\wedge g$ is the maximal common initial segment $s\in\twolo$ of $f$ and $g$.

Under the name of neat families, tree-like almost disjoint families were first used in \cite[II.4.1]{Sh:Proper} and  \cite{Ve:OCA}. 

\begin{definition} \label{Def.tree-like}
	A family $\A$  of almost disjoint sets of 
	integers is  \emph{tree-like} if 
	there is an ordering $\prec$ on its domain $\D=\bigcup\A$ such that 
	$\<\D,\prec\>$ is a  tree of height~$\omega$ and each element of $\A$ is
	included in a unique  maximal branch of this tree. 
\end{definition}

An example of a tree-like family that is also a perfect subset of $\cPN$ is given as follows. 

\begin{example}[Perfect tree-like almost disjoint family] \label{ex.perfect} 
	Suppose that $J(s)$, for $s\in \twolo$, are pairwise disjoint nonempty finite subsets of $\bbN$. For $f\in \twoo$ let 
	\[
	J(f)=\bigcup\{J_{f\rs n}\mid f\in\twoo\}. 
	\]
	Then $A(f)\cap A(g)=\bigcup_{s\sqsubseteq f\wedge g} J(s)$ for $f\neq g$. Therefore the family 
	\[
\cA\{J_s\}=\{J(f)\mid f\in \twoo\}
	\]
	is  tree-like, and even a perfect subset of $\cPN$. 

We say that $\cA\{J_s\}$ is a \emph{perfect tree-like almost disjoint family}. 
\end{example}

\begin{lemma} \label{L.perfect.coarsening} If $\cA$ is a perfect tree-like almost disjoint family, then there is a perfect tree-like almost disjoint family $\cB$ such that every element of $\cB$ includes $2^{\aleph_0}$ elements of $\cA$. 
\end{lemma}

\begin{proof}
We have  $\cA=\cA\{J_s\}$ for pairwise disjoint finite sets $J_s$, for $s\in \twolo$. 
For $s\in \twolo$ with length $n$, let 
\[
I_s=\bigcup\{J_t\mid 2n \leq |t|<2(n+1), (\forall j< n) t(2j+1)=s(j)\}.
\]
This is a family of finite, pairwise disjoint, sets. 
Let $\cB=\cA\{I_s\}$. Then every $f$ satisfies
\[
I(f)=\bigcup_{n} I_{f\rs n}=\bigcup\{J(g)\mid g(2j+1)=f(j)\text{ for all }j\}
\]
and $\cB$ is as required. 
\end{proof}

\subsection{Martin's Axiom}\label{S.MA.ad}

We assume that the reader is familiar with Martin's Axiom $\MA$ and its weakening $\MAsigma$. See for example \cite{Ku:Set}.

For distinct $a$ and $b$ in $\twoo$ let\index[symbol]{$\Delta(a,b), a\wedge b$} 
\begin{align*}
	\Delta(a,b)&=\min\{n\mid a(n)\neq b(n)\},\\
	a\wedge b&=a\rs \Delta(a,b). 
\end{align*}
If $Z\subseteq \twoo$ then let
\(
\Delta(Z)=\{a\wedge b\mid a,b\text{ are distinct elements of $Z$}\}. 
\)
Equivalently,  
$\Delta(Z)=\{s\in \twolo\mid [s^\frown j]\cap Z\neq \emptyset\text{ for }j=0,1\}$. 
The set $\Delta(Z)$ is a tree with respect to $\sqsubseteq$ and its $i$-th level is $\{s\in \Delta(Z)\mid |\{t\in \Delta(Z)\mid t\sqsubset s\}|=i+1\}\}$. 

\begin{lemma} \label{L.MAsl.thin}
	Assume $\MAsigma$. Also assume $Z\subseteq \twoo$ is uncountable and  $I(s)$ is a finite interval in $\bbN$  such that $|s|\leq \min(I(s))$ for every $s\in \Delta(Z)$.  Then there are an uncountable $Z'\subseteq Z$, an increasing sequence $k_i$, for $i\in \bbN$, in~$\bbN$, and $i\colon \Delta(Z')\to \bbN$  such that every $s\in \Delta(Z')$ satisfies $I(s)\subseteq [k_{i(s)},k_{i(s)+1}))$. Moreover, for every $i\in \bbN$, the set  $\cS_i=\{s\in \Delta(Z')\mid i(s)=i\}$ is the $i$-th level of $\Delta(Z')$. 
\end{lemma}

\begin{proof} We may assume that $Z$ has no isolated points, by removing all relatively open and countable subsets of $Z$. Let $\bbP$ be the poset of all $p=(F(p),\vec k(p))$ such that $F(p)\Subset Z$ and $\vec k(p)$ is a tuple $k_0(p)<k_1(p)<\dots< k_{l(p)}(p)$ for some $l(p)\in \bbN$ such that the following holds.
	\begin{enumerate}
		\item For all $s\in \Delta(F(p))$, $I(s)\subseteq [k_i(p), k_{i+1}(p))$ for some $i<l(p)$. 
		\item For all $i<l(p)$, $|\{s\in \Delta(F(p))\mid I(s)\subseteq [k_i(p), k_{i+1}(p))\}$ is the $i$-th level of $\Delta(F(p))$. 
	\end{enumerate}
	In order to prove that $\bbP$ is $\sigma$-linked, it suffices to prove that every two conditions $p$ and $q$ such that $\vec k(p)=\vec k(q)$ and (with $l=l(p)=l(q)$ and  $k=k_l(p)=k_l(q)$) they satisfy $\{z\rs k\mid z\in F(p)\}=\{z\rs k \mid z\in F(q)\}$ are compatible. 
	
	Assume for a moment that $F(p)\cap F(q)=\emptyset$. Thus for every $x\in F(p)$ the unique $x'\in F(q)$ such that $x\rs k_l=x'\rs k_l$ is distinct from $x$. 
	Then let $F(r)=F(p)\cup F(q)$, $l(r)=l+1$, $k_i(r)=k_i(p)$ for $i\leq l$, and $k_{l+1}(r)=\max\{I(\Delta(x,x')) \mid x\in F(p)\}+1$. 
	Since $I(x\wedge x')\geq |x\wedge x'|\geq k$,  	the condition $r=(F^r,\vec k(r))$ belongs to $\bbP$ and extends both $p$ and $q$. 
	
	Now consider the case when the set $F_0=F(p)\cap F(q)$  is nonempty. Thus $F_0$ is the set of all $x\in F(p)$ such that the unique $x'\in F(q)$ which satisfies $x\rs k_l=x'\rs k_l$ is equal to $x$. Since $Z$ has no isolated points, for every $x\in F_0$ we can find $x'\in Z\setminus \{x\}$ such that $x\rs k=x'\rs k$. Let $F(r)=F(p)\cup F(q)\cup \{x'\mid x\in F_0\}$, $l(r)=l+1$, $k_i(r)=k_i(p)$ for $i\leq l$, and $k_{l+1}(r)=\max\{I(\Delta(x,x')) \mid x\in F(p)\}+1$. 
	As in the previous case, $r=(F(r),\vec k(r))$	 is a condition that extends both $p$ and $q$. 
	
		For every $z\in Z$ the condition $p_z$ satisfying $F(p_z)=\{z\}$, $l(p_z)=0$ and $k_0(p_z)=0$ belongs to $\bbP$. 
		Since $\bbP$ is $\sigma$-linked, some condition $p\in \bbP$ forces that $\{z\in Z\mid p_z\in G\}$ is uncountable. 
	Therefore $\MAsigma$ implies that for some filter $G\subseteq \bbP$ the set $Z'=\{z\in Z\mid p_z\in G\}$ is uncountable and that for every $l$ there is $p\in G$ satisfying $l(p)\geq l$.  Since $G$ is a filter, for every $F\Subset Z$ there is $p\in G$ such that $F(p)\supseteq F$. For each $i$ let $k_i=k_i(p)$ for some $p\in G$ such that $l(p)\geq l$. Since $G$ is a filter, $k_i$ is well-defined. Clearly $Z'$ and the sequence $k_i$ are as required. 
\end{proof}

\subsection{Open colouring axioms}\label{S.OCAsharp}

In this section we state two apparent strengthenings of $\OCAT$ each one of which is equivalent to it.  It is the introduction of the second one, $\OCAsharp$, in \cite{de2023trivial} that precipitated the progress reported here.  Readers unfamiliar with these axioms may want to skip this subsection on the first reading, and consult the definitions later on as needed. 

If $X$ is a topological space and $U$ is a subset of $X^2$ disjoint from the diagonal that is symmetric (i.e., $(x,y)\in U$ if and only if $(y,x)\in X$) then we $U$ corresponds to a unique subset of $[X]^2$, the space of unordered pairs of elements of $X$, and every subset of $[X]^2$ corresponds to symmetric $U\subseteq X^2$ disjoint from the diagonal.  We say that a subset of $[X]^2$ is \emph{open} if the corresponding $U$ is open in the product topology.

\begin{definition} 
	$\OCAT$ is the following statement. 
	If $X$ is a separable metric space  and $[X]^2=K_0\cup K_1$ is a partition such that $K_0$ is open, then $X$ either has  an uncountable $K_0$-homogeneous subset
	or it  can be covered by a countable family of $K_1$-homogeneous sets. 
\end{definition}

\begin{definition}\label{Def.OCAinfty} 
	$\OCAinfty$ is the following statement. 
	If $X$ is a separable metric space  and $[X]^2=K_0^n\cup K_1^n$, for $n\in \bbN$  are open
	partitions such that $K_0^n\supseteq K_0^{n+1}$ for all $n$ then one of the following applies
	\begin{enumerate}
		\item There are sets $X_n$, for $n\in \bbN$, such that $X=\bigcup_n X_n$ and $[X_n]^\subseteq K_1^n$. (We say that $X$ is $\sigma$-$K_1^*$-homogeneous.)
		\item There are an uncountable $Z\subseteq \twoo$ and a continuous $f\colon Z\to X$  such that $\{f(a),f(b)\}\in K_0^{\Delta(a,b)}$ for all $a,b$ in $Z$. 
	\end{enumerate}
\end{definition}

Although $\OCAsharp$ appears to be tailor-made to work with Biba’s trick (see the final segment of the proof of Theorem~\ref{T.OCAsharp-Fin}  and the final segment of the proof of Proposition~\ref{P.uniformization}), realizing this took considerably more time than the author would be prepared to admit. 

\begin{definition} \label{Def.OCAsharp} 
	$\OCAsharp$ is the following statement. Suppose that $X$ is a separable metric space and that $\cV_j$ is a countable family of symmetric open subsets of $[X]^2$ such that $\bigcup\cV_j\supseteq \bigcup\cV_{j+1}$ for all $j\in \bbN$. 
	Then one of the following alternatives holds. 
	\begin{enumerate}
		\item \label{OCAsharp.1} There are sets $X_n$, for $n\in \bbN$, such that $X=\bigcup_n X_n$ and 
		\[
		[X_n]^2\cap\bigcup\cV_n =\emptyset
		\] 
		for all $n$. 
		\item\label{OCAsharp.2} There are an uncountable $Z\subseteq \twoo$, an injective $f\colon Z\to X$, and $\rho\colon \Delta(Z)\to \bigcup_j \cV_j$  such that $\rho(s)\in \cV_{|s|}$ for all $s$ and all distinct $a$ and $b$ in $Z$ satisfy
		\[
		\{f(a),f(b)\}\in \rho(a\wedge b). 
		\]
	\end{enumerate}
\end{definition}

In \cite[\S 5]{moore2021some} it was proven that $\OCAT$ and $\OCAinfty$ are equivalent, and analogous proof was used in \cite[Theorem~3.3]{de2023trivial} to prove that $\OCAsharp$ is equivalent to $\OCAT$.

\section{Approximations to a homomorphism}

\subsection{Coherence of approximations}
\label{S.Coherence}
\begin{definition}
Suppose that $\cI$ is an ideal on $\bbN$, $\Phi\colon \cP(\bbN)\to \cP(\bbN)/\cI$ is a homomorphism, $\Phi_*$ is a lifting of $\Phi$,  and $\cK$ is an approximation to $\cI$ (i.e., $\cI\subseteq \cK\ucup \Fin$). A function $\Theta\colon \cPN\to \cPN$ is a \emph{$\cK$-approximation to $\Phi$} if for every $A\subseteq \bbN$ we have 
\[
\Phi_*(A)\Delta \Theta(A)\in \cK\ucup \Fin. 
\]
If this holds for all $A\in \cX$ for some $\cX\subseteq \cPN$, then we say that $\Theta$ is a \emph{$\cK$-approximation to $\Phi$ on $\cX$}. 
\end{definition}

The following variation on the standard stabilization trick is a second-order corollary (it is a corollary to Corollary~\ref{C.HereditaryComeager}). 

\begin{coro}\label{C.Stabilization}   Suppose that $\cI$ is an ideal on $\bbN$ with closed approximation $\cK$, $\Phi\colon \cP(\bbN)\to \cP(\bbN)/\cI$ is a homomorphism, and $\calH$ is a hereditary nonmeager subset of $\cPN$, 
	\begin{enumerate}
		\item 	If  $\Phi$ has a Baire-measurable $\cK$ approximation on $\calH^2$, then it has a continuous $\cK\ucup \cK$ approximation on $\calH$. 
		\item 	If  $\Phi$ has a $\cK$ approximation on $\calH^2$ whose graph can be covered by graphs of countably many Baire-measurable functions, then it has a $\cK\ucup \cK$ approximation whose graph can be covered by graphs of countably many continuous functions on $\calH$. 
	\end{enumerate}
\end{coro}

\begin{proof}
	For the first part, suppose $\Theta$ is a Baire-measurable $\cK$-approximation to $\Phi$ on $\calH^2$. 
	Every Baire-measurable function between Polish spaces is continuous on a comeager set $\cX$. By Corollary~\ref{C.HereditaryComeager},
	there are a partition $\bbN=A_0\sqcup A_1$ and sets $C_0\subseteq A_0$ and $C_1\subseteq A_1$ such that for every $X\in \calH$ both $(X\cap A_0)\cup C_1$ and $(X\cap A_1)\cup C_0$ belong to $\calH^2\cap \cX$. 
	Then 
	\[
	X\mapsto \Theta((X\cap A_0)\cup C_1)\cap \Phi_*(A_0)\cup \Theta((X\cap A_1)\cup C_0)\cap \Phi_*(A_1) 
	\]
	is a continuous $\cK^2$-approximation to $\Phi$ on $\calH$. 
	
	For the second part, fix  Baire-measurable functions $\{\Theta_m\}$ whose graphs cover the graph of a $\cK$-approximation to $\Phi$ are continuous. Apply Corollary~\ref{C.HereditaryComeager} as before. For all $m$ and $n$, the function  
	\[
	X\mapsto \Theta_m((X\cap A_0)\cup C_1)\cap \Phi_*(A_0)\cup \Theta_n((X\cap A_1)\cup C_0)\cap \Phi_*(A_1) 
	\]
	is continuous, and graphs of these functions cover the graph of a $\cK^2$-approximation to $\Phi$ on $\calH$. 
\end{proof}

\begin{lemma}\label{P.Theta0Theta1} Suppose that $\cI$ is an analytic ideal, that $\Phi\colon \cPN\to \cPN/\cI$ is a homomorphism, and that $\calH_j$ for $j=0,1$ is a nonmeager hereditary subset of $\cPN$ closed under finite changes of its elements. 
	\begin{enumerate}
		\item 	If $\Theta_i\colon \cPN\to \cPN$ is a continuous  lifting of $\Phi$ on $\calH_i$ for $i=0,1$,  then each $\Theta_i$ is a lifting of $\Phi$ on a relatively comeager  subset of $\calH_0\cup \calH_1$. 
		\item 	If $\cK$ is an analytic approximation to $\cI$ and $\Theta_i\colon \cPN\to \cPN$ is a continuous  $\cK$-approximation to $\Phi$ on $\calH_i$ for $i=0,1$,  then each $\Theta_i$ is a $\cK^2$-approximation to $\Phi$ on a relatively comeager  subset of $\calH_0\cup \calH_1$. 
	\end{enumerate}
\end{lemma}

\begin{proof}By taking $\cK=\cI$ and and noting that in this case $\cK^2=\cK$, one sees that it suffices to prove the second part. Towards this,  note that the set 
	\[
	\cY=\{A\mid \Psi_0(A)\Delta \Psi_1(A)\in \cK\ucup \Fin\}
	\] 
	is Borel.
	By Theorem~\ref{T.HNM}, $\calH=\calH_0\cap \calH_1$ is hereditary and nonmeager.  Since $\cY$ is analytic and includes $\calH$, it is comeager.  Hence for a relatively comeager set of $A$ in $\calH_0$ we have that  $\Theta_1(A)\Delta \Phi_*(A)\subseteq (\Theta_1(A)\Delta \Theta_0(A))\cup (\Theta_0(A)\Delta \Phi_*(A))$ belongs to $ \cK^2\ucup \Fin$. Therefore~$\Theta_1$ is a $\cK^2$-approximation to $\Phi$ on a relatively comeager subset of $\calH_0$. An analogous proof shows that $\Theta_0$ is a $\cK^2$-approximation to $\Phi$ on a relatively comeager subset of~$\calH_1$, and the conclusion follows. 
\end{proof}

 The following coherence property is a more precise version of Lemma~\ref{P.Theta0Theta1} (see also Claim~\ref{L.hA.coherent} from the proof of Theorem~\ref{T.OCAsharp-Fin} for a simpler version).

\begin{lemma}
	\label{L.Xstabilizer}
	Suppose that $\cI$ is an ideal, $\Phi\colon \cP(\bbN)\to \cP(\bbN)/\cI$ is a homomorphism,~$\cK$ is a closed approximation to $\cI$, and $\Psi_i$ is a continuous $\cK$-approximation to $\Phi$ on a nonmeager hereditary  set $\calH_i$, for $i=0,1$. 
	Then there are $k$ and  $m'$ in $\bbN$ such that all $A\subseteq[k,\infty)$  satisfy $(\Psi_0(A)\Delta \Psi_1(A))\setminus m'\in \cK^{10}$. 
\end{lemma}

\begin{proof} By Theorem~\ref{T.HNM}, $\calH=\calH_0\cap \calH_1$ is hereditary and nonmeager.  We claim that there is $m’$ such that all $A\in \calH$ satisfy $(\Psi_0(A)\Delta \Psi_1(A))\setminus m'\in \cK^6$. 
	Since both $\Psi_0$ and $\Psi_1$ are $\cK$-approximations to $\Phi$, every $a\in \calH$ satisfies $\Psi_0(A)\Delta \Psi_1(A)\in \cK^2\ucup \Fin$. Let 
	\[
	g(A)=\min\{m\vert (\Psi_0(A)\Delta \Psi_1(A))\setminus m\in \cK^2\}. 
	\]
	By the continuity of $\Psi_0$ and $\Psi_1$, the set 
	\[
	\cX_m=\{A\in \calH\vert g(A)\leq m\}
	\] 
	is relatively closed in $\calH$ for every $m$. Since $\calH$ is nonmeager, we can find $m$ large enough for~$\cX_m$ to have nonempty interior (relative to $\calH$). Let $k$ and $u\subseteq k$ be such that the clopen set $U=\{A\mid A\cap k=u\}$ satisfies $U\cap \calH\subseteq \cX_m$.  
	Then  all $A\subseteq [k,\infty)$ satisfy 
	\begin{equation}\label{eq.Psi01onU}
	(\Psi_0(u\cup A)\Delta \Psi_1(u\cup A))\setminus m \in \cK^2. 
	\end{equation}
	By Lemma~\ref{L.HNM.1}  and increasing $k$ if needed we may assume that $\calH\cap \cP([k,\infty))$ is dense in $\cP([k,\infty))$.

	Since $\Psi_0$ is a $\cK$-approximation to the homomorphism $\Phi$ on $\calH$, for every $A\in \calH$ and every $s\in \Fin$ we have $\Psi_0(A)\Delta \Psi_0(A\Delta s)\Delta\Psi_0(s)\in \cK^3\ucup \Fin$. Thus we can set 
	\[
	f_{0}(A)=\min\{l\vert (\Psi_0(A)\Delta \Psi_0(A\Delta s)\Delta \Psi_0(s))\setminus l\in \cK^3\text{ for all }s\subseteq k\}. 
	\]
	Let $f_{1}\colon \calH\to \bbN$ be the analogously defined  function with $\Psi_0$ replaced with $\Psi_1$.
	The sets $\{A\in\calH\vert f_{j}(A)>n\}$ for $j=0,1$ are, by the continuity of $\Psi_j$, relatively open for every~$n$. 
		Since $\calH$ is nonmeager, 
	\[
	\cZ=\{a\in \calH\vert \max(f_0(a),f_1(a))\leq m'\}
	\]
	has nonempty interior for some 
	$m'\geq m$.   We claim that $m'$ is as required. Fix $A\in \calH$ such that $\min(A)\geq k$. Then (writing $x=^{l,\calL} y$ if $(x\Delta y)\setminus l\in \calL$ and applying \eqref{eq.Psi01onU} to $A$ and to $u\cup A=u\Delta A$ in the second equality)
	\[
	\Psi_0(A)=^{m',\cK^3}\Psi_0(u\Delta A)\Delta \Psi_0(s) =^{m',\cK^4} 
	\Psi_1(u\Delta A)\Delta \Psi_1(s) =^{m',\cK^3}
	\Psi_1(A)
	\]
	thus $(\Psi_0(A)\Delta \Psi_1(A))\setminus m'\in \cK^{10}$ as promised. 
		Since $\calH\cap \cP([k,\infty))$ is dense in $\cP([k,\infty))$ and both $\Psi_0$ and $\Psi_1$ are continuous, this holds for all $A\subseteq [k,\infty)$, as required. 
\end{proof}

\subsection{Ideals associated with approximations to a homomorphism $\Phi$}

The proof of Theorem~\ref{T.OCA-l}  can be described as a struggle to prove that the ideals defined in Definition~\ref{Def.J1JcontK} below are large in the appropriate sense. This method  goes back to the original rigidity proof, that in an oracle-cc forcing extension all automorphisms of $\cPN/\Fin$ are trivial (\cite{Sh:Proper}). In this proof, one shows that $\Jcont$ (as defined below) is dense (in the sense that every infinite $A\subseteq \bbN$ has an infinite subset that belongs to $\Jcont$), then that it is a P-ideal, and finally that it is equal to $\cPN$. Being a P-ideal is used to assure that a certain poset is ccc. Courtesy of Biba, we can skip this part. Instead we prove that $\Jcont$ intersects every perfect tree-like almost disjoint family (Proposition~\ref{P.general.Jcont.nonmeager}) and then that it is equal to $\cPN$ (or, in case when $\Phi$ is a homomorphism, that a single continuous function provides a lifting of $\Phi$ on all sets in $\Jcont$).

\begin{definition} \label{Def.J1JcontK} If $\cI$ is an ideal on $\bbN$ and is $\cK$ is a closed approximation to $\cI$, then for a homomorphism $\Phi\colon \cP(\bbN)\to \cP(\bbN)/\cI$ let
	\begin{align*}
		\Jcont^\cK(\Phi)=\{A\vert &\Phi\text{ has a continuous $\cK$-approximation on $\cP(A)$}\},\\
		\Jcontnm^\cK(\Phi)=\{A\vert &\Phi\text{ has a continuous $\cK$-approximation}\\ 
		&\text{on a relatively nonmeager hereditary subset of $\cP(A)$}\}, \\ 
		\cJ_\sigma^\cK(\Phi)=\{A\vert &\Phi\text{ has a $\cK$-approximation on $\cP(A)$ whose graph}\\
		&\text{can be covered by graphs of countably many Borel functions}\}. 
	\end{align*}
	We will omit the parameter $\Phi$ and write $\Jcont^\cK$, $\Jcontnm^\cK$, or $\Jsigma^\cK$  whenever $\Phi$ is clear from the context (that would be always, since all homomorphisms in this paper are denoted $\Phi$).
We also write 
\[
\Jcont=\Jcont^\cI,\quad  \Jcontnm=\Jcontnm^\cI, \quad \text {and}\quad \Jsigma=\Jsigma^\cI. 
\]
\end{definition}

\begin{lemma}\label{L.JcontJ1JsigmaK}
	Suppose that $\cI$ is an ideal with approximations $\cK$ and $\calL$ and that $\Phi\colon \p(\N)\to \p(\N)/\Fin$
	is a  homomorphism
	
	\begin{enumerate}
		\item \label{L.Jcont.1K} We have $\cJ^\cK\ucup \cJ^\calL\subseteq \cJ^{\cK\ucup \calL}$, $\Jcontnm^\cK\ucup \Jcontnm^\calL\subseteq \Jcontnm^{\cK\ucup \calL}$, $\cJ_\sigma^\cK\ucup \cJ_\sigma^\calL\subseteq \cJ_\sigma^{\cK\ucup\calL}$, and $\Fin\subseteq \Jcont^\cK\subseteq  \Jcontnm^\cK$. 
		\item \label{L.Jcont.4K} Each one of $\Jcont^\cK$, $\Jcontnm^\cK$, and $\Jsigma^\cK$ is closed under finite changes of its elements. 
	\end{enumerate} 
\end{lemma}

\begin{proof} \eqref{L.Jcont.1K} Observe that if $\Theta$ is a $\cK$-approximation to $\Phi$ on $\cP(A)$ and $\Upsilon$ is an $\calL$-approximation to $\Phi$ on $\cP(B)$, then 
	\[
	X\mapsto \Theta(X\cap A)\cup \Upsilon ((X\setminus A)\cap B)
	\]
is a $\cK\ucup \calL$-approximation to $\Phi$ on $A\cup B$.

\eqref{L.Jcont.4K} would have been  trivial had we only assumed that $\ker(\Phi)\supseteq \Fin$. If $A\in \Jcont^\cK$ as witnessed by $\Theta$  and $s\Subset \bbN\setminus A$, fix a lifting $F$ of $\Phi$ on $\cP(s)$. Then $X\mapsto F(X\cap s)\cup \Theta(X\cap A))$ is a $\cK$-approximation to $\Phi$ on $\cP(A\cup s)$. Proofs in case of $\Jcontnm^\cK$ and $\Jsigma^\cK$  are analogous. 
\end{proof}

The following is an immediate consequence of Corollary~\ref{C.Stabilization}.  

\begin{lemma} \label{L.K-baire-to-ctns}
	Suppose that  $\cI$ is an ideal with approximation $\cK$ and  $\Phi\colon \p(\N)\to \p(\N)/\cI$
	is a  homomorphism.
	\begin{enumerate}
		\item $\{A\mid \Phi\rs \cP(A)$ has a Baire-measurable $\cK$-approximation$\}\subseteq \Jcont^{\cK^2}$. 
		\item $\{A\mid \Phi\rs \cP(A)$ has a  $\cK$-approximation whose graph can be covered by graphs of countably many Baire-measurable functions$\}\subseteq\Jsigma^{\cK^2}$. \qed
	\end{enumerate}
\end{lemma}

%

The following proposition is a variation on \cite[Lemma 3.12.3]{Fa:AQ} (see also \cite[Lemma~3C]{Fr:Farah} and the appendix to \cite{de2023trivial}) and we include its proof for completeness.

\begin{proposition}
	\label{P+.Jsigma-to-Jcont}	\label{P.Jsigma-to-Jcont}
	Suppose that $\cI$ is an ideal with a closed approximation~$\cK$ and  $\Phi\colon \cP(\bbN)\to \cP(\bbN)/\cI$ is a homomorphism. If $X\in \Jsigma(\Phi)$ and $X=\bigsqcup_n A_n$,   then $A_n\in \Jcont^{\cK^2}$  for all but finitely many $n$. 
\end{proposition}

The proof of Proposition~\ref{P+.Jsigma-to-Jcont} uses the following lemma. 

\begin{lemma}\label{L+.Jsigma-to-Jcont}
	Suppose that $\cI$ is an ideal with a closed approximation~$\cK$, 	$\Phi\colon \cP(\bbN)\to \cP(\bbN)/\cI$ is a homomorphism with lifting $\Phi_*$, there is a partition $\bbN=A\sqcup B$, $[s]$ is a relatively clopen subset of $\cP(B)$,  and $F\colon \cP(\bbN)\to \cP(\bbN)$ is Borel-measurable.  Then the following holds. 
	\begin{enumerate}
		\item \label{L.Jsigma-to-Jcont.1}	The restriction of $\Phi$  to  the set $\cT$ of all $a\subseteq A$ such that the set 
		\[
		\cZ(a)=\{b\in [s]\cap \cP(B)\mid F(a\cup b)\cap \Phi_*(A)\Delta \Phi_*(a)\in \cK\ucup \Fin\}
		\]
		is comeager in $\cP(B)$ has a C-measurable $\cK$-approximation.
		
		\item  \label{L.Jsigma-to-Jcont.2} In particular, if $\cT$ is relatively comeager in some clopen subset of $\cP(A)$, then~$\Phi$ has a continuous $\cK^2$-approximation on $\cP(A)$. 
	\end{enumerate}	
\end{lemma}

\begin{proof} \eqref{L.Jsigma-to-Jcont.1} For simplicity of notation, we may assume $[s]=\cP(B)$. 
	Since the Boolean operations $\cap$ and $\Delta$ are continuous, the function 
	\[
	(a,b,c)\mapsto (F(a\cup b)\cap \Phi_*(A))\Delta c
	\] 
	is Borel, and therefore the set 
	\[
	\cX=\{(a,b,c)\in \cP(A)\times \cP(B)\times \cP(\bbN)\mid (F(a\cup b)\cap \Phi_*(A))\Delta c\in \cK\ucup \Fin\}
	\]
	is, being the preimage of $\cK\ucup \Fin$ by a Borel-measurable function, Borel-measurable. 
	The set
	\[
	\cY=\{(a,c)\in \cP(A)\times \cP(B)\mid \{b\subseteq B\mid (a,b,c)\in \cX\}\text{ is comeager}\}
	\]
	is, by Novikov's theorem (\cite[Theorem~29.3]{Ke:Classical}), analytic. By our assumption, for every $a\in \cT$ the set $\cZ(a)=\{b\subseteq B \mid (a,b,\Phi_*(a))\in \cX\}$ is comeager in $\cP(B)$, in particular the section $\cY_a$ is nonempty for all $a\in \cT$. 
	
	Therefore the Jankov, von Neumann Uniformization theorem (\cite[18.A]{Ke:Classical}) implies that there exists a C-measurable function 
	$
	G_0\colon \cP(\A)\to \cP(B)
	$
	such that $(a,G_0(a))\in \cY$ for all $a\in \cT$. 
	Let 
	\[
	G(a)=F(a\cup G_0(a))\cap \Phi_*(a). 
	\]
	Then $G$ is the composition of a C-measurable function and a Borel function. Since the preimage of an analytic set by a Borel-measurable function is analytic, the $G$-preimage of every open set  is  C-measurable, and $G$ is therefore C-measurable. Moreover, for every $a\in \cT$ we have that $G_0(a)$ belongs to the comeager set $\cZ(a)$, and therefore $G(a)\Delta \Phi_*(a)\in \cK\ucup\Fin$, as required.  
	
	\eqref{L.Jsigma-to-Jcont.2} Assume that $\cT$ is relatively comeager in some clopen subset of $\cP(A)$. By~\eqref{L.Jsigma-to-Jcont.1},~$\Phi$ has a C-measurable $\cK$-approximation on $\cT$. Since every C-measurable function is Baire-measurable, $\Phi$ has a continuous $\cK^2$-approximation on $\cP(A)$ by Lemma~\ref{L.JcontJ1JsigmaK}. 
\end{proof}

\begin{proof}[Proof of Proposition~\ref{P+.Jsigma-to-Jcont}]We may assume $X=\bbN$. 	Fix a partition $\bbN=\bigsqcup_n A_n$ and Borel-mea\-su\-ra\-ble functions $F_n\colon \cP(\bbN)\to \cP(\bbN)$ whose graphs cover the graph of a $\cK$-ap\-proxi\-ma\-ti\-on to $\Phi$. 
	
	It suffices to prove that the restriction of $\Phi$ to $\cP(A_n)$ has a continuous $\cK^2$-approximation. Assume this is not the case. Since the ideal $\Jcont$ is closed under finite changes of its elements (Lemma~\ref{L.JcontJ1JsigmaK}), this implies that for every $n$ and every nonempty clopen subset $[t]$ of $A_n$, the restriction of $\Phi$  to $[t]\cap \cP(A_n)$ has no continuous $\cK$-measurable approximation. 
		It will be convenient to write 
	\[
	\textstyle	C_n=\bigcup_{j>n}A_n.	
	\]
	We will recursively choose $a_n\subseteq A_n$ and  $\cX_n\subseteq \cP(C_n)$ such that the following conditions hold for all $n$. 
	\begin{enumerate}[label = (\roman*)]
		\item\label{I+.sigma-Borel.1}  $(F_n(\bigcup_{j\leq n} a_j\cup b)\cap \Phi_*(A_n))\Delta  \Phi_*(a_n)\notin \cK\ucup\Fin$ for all $b\in \cX_n$.  
		\item $\cX_n$ is C-measurable and relatively comeager in $[s_n]\cap \cP(C_n)$ for some clopen $[s_n]\subseteq \cP(C_n)$. 
		\item $\{a_n\}\times \cX_{n+1}\subseteq \cX_n$. 
	\end{enumerate}
	We will describe the selection of $a_n$ and $\cX_n$. 
	For $n=0$, our assumption that $\Phi$ has no continuous $\cK^2$-approximation on $\cP(A_0)$ and Lemma~\ref{L+.Jsigma-to-Jcont} together imply that for some $a_0\subseteq A_0$ the set $\cX_0'$ of all $x\subseteq C_0$ such that $F_0(a_0\cup x)\cap \Phi_*(A_0)\Delta \Phi_*(a_0)\notin \cK^2$ is nonmeager. 
	Since this set is, as a preimage of a Borel set by a Borel-measurable function, Borel-measurable, there is a clopen set $[s_0]\subseteq \cP(C_0)$ such that $\cX_0=[s_0]\cap \cX_0'$ is relatively comeager in $[s_0]\cap \cP(C_0)$. 
	
	This describes the construction of $a_0$, $\cX_0$, and $[s_0]$. 
	
	Suppose that $a_n$, $\cX_n$, and $[s_n]$ as required had been chosen. 
	Then we can write $[s_n]=[t_n]\times [u_n]$ for clopen sets $[t_n]\subseteq \cP(A_{n+1})$ and $[u_n]\subseteq \cP(C_{n+1})$. Since $\cX_n$ is relatively comeager in $[s_n]\cap \cP(C_n)$,  by the Kuratowski--Ulam theorem (\cite[\S 8.K]{Ke:Classical}), the set 
	\begin{multline*}
		\cT_n=\{a\in [t_n]\cap \cP(A_{n+1})\mid \text{the set }\{x\subseteq \cP(C_{n+1})\mid  a\cup x\in \cX_n\}\\ \text{ is relatively comeager in $[u_n]\cap \cP(C_{n+1})$}\}
	\end{multline*}
	is relatively comeager in $[t_n]\cap \cP(A_{n+1})$. 
	
	Since the intersection of two comeager sets is comeager, by the fact that $\Phi$ has no continuous $\cK^2$-approximation on $[t_n]\cap \cP(A_n)$, and by applying Lemma~\ref{L+.Jsigma-to-Jcont} to $F_{n+1}$,  $\cT_n$, and $[u_n]$,  we can find $a_{n+1}\subseteq [t_n]\cap \cP(A_{n+1})$ such that 
	\begin{multline*}
		\cX_{n+1}=\{x\in [u_n]\cap \cP(C_{n+1})\mid a_{n+1}\cup  x\in \cX_n,\\ F_{n+1}(a_{n+1}\cup x)\cap \Phi_*(A_{n+1})\Delta  \Phi_*(a_{n+1})\notin \cK\}
	\end{multline*}
	is nonmeager in $\cP(C_{n+1})$. Being C-measurable, $\cX_{n+1}$ is relatively comeager in $\cP(C_{n+1})\cap [s_{n+1}]$ for some relatively clopen $[s_{n+1}]\subseteq [u_n]\cap \cP(C_{n+1})$.  
	
	Then $\{a_{n+1}\}\times \cX_{n+1}\subseteq \cX_n$, and the sets $a_{n+1}$, $\cX_{n+1}$, and $[s_{n+1}]$ satisfy the requirements. 
	
	This describes the recursive construction of $a_n$, for $n\in \bbN$.  Let $a=\bigcup_n a_n$. By the assumption, $F_n(a)\Delta  \Phi_*(a)\in \cK\ucup \Fin$ for some $n$. However, $a\cap A_n=a_n$, hence $(F_n(a)\cap \Phi_*(A_n))\Delta  \Phi_*(a_n)\in \cK\ucup \Fin$,  but this contradicts~\eqref{I+.sigma-Borel.1}. 
	
	Therefore the assumption that the restriction of $\Phi$ to $\cP(A_n)$ does not have a continuous $\cK^2$-approximation is false, and $A_n\in \Jcont^{\cK^2}(\Phi)$ for some $n$.  Since this applies to any sequence of sets whose union is in $\Jsigma^\cK(\Phi)$, we conclude that $A_n\in \Jcont^{\cK^2}(\Phi)$ for all but finitely many $n$. 
\end{proof}

\section{Local liftings}

At the end of this section we will prove the following local version of Theorem~\ref{T.OCA-l}. 

\begin{proposition} 
	\label{P.general.Jcont.nonmeager}
	Assume  $\OCAT$.
	If an ideal $\cI$ is countably $16$-determined   and $\Phi\colon \p(\N)\to \p(\N)/\cI$
	is a  homomorphism, then the ideal $\Jcont(\Phi)$ intersects every perfect tree-like almost disjoint family nontrivially. 
\end{proposition}

Definition~\ref{Def.PartitionXA} has a long history, starting with \cite{Ve:OCA}. A function $\Phi_*$ will typically be a lifting of a homomorphism from $\cP(\bbN)$ into $\cP(\bbN)/\cI$.

\begin{definition}\label{Def.PartitionXA}
	Suppose that $\cA$ is  a tree-like almost disjoint family. By $\hat \cA$ we denote the hereditary closure of $\cA$ and for an infinite $B\in \hat\cA$ write $A(B)$\Sindex{$A(B)$} for the unique element of $\cA$ that includes $B$. Let 
	\[
	\cX_\cA=\{(C,B)\mid B\in \hat \cA, C\subseteq B\}.
	\] 
	For $x=(C,B)$ in $ \cX_\cA$ we write $C=C(x)$, $B=B(x)$, and $A(B(x))=A(x)$.
	
		If in addition $\Phi_*\colon \cP(\bbN)\to \cP(\bbN)$ and $\cK$ is a closed hereditary subset of $\cP(\bbN)$, then we define a partition\Sindex{$[\cX_\cA]^2=K^{\Phi_*,\cA,\cK}_0\cup K^{\Phi_*,\cA,\cK}_1$}
	\[
	[\cX_\cA]^2=K^{\Phi_*,\cA,\cK}_0\cup K^{\Phi_*,\cA,\cK}_1
	\]
	by setting $\{x,y\}\in K^{\Phi_*,\cA,\cK}_0$ if the following three conditions hold. 
	\begin{enumerate}[label = $K_0$(\roman*)]
	\item \label{K0.1} $A(x)\neq A(y)$.
	\item \label{K0.2}  $B(x)\cap C(y)=C(x)\cap B(y)$.
	\item \label{K0.3}  $(\Phi_*(B(x))\cap \Phi_*(C(y)))\Delta (\Phi_*(C(x))\cap
	\Phi_*(B(y)))\notin \cK^2$.
\end{enumerate}
Endow $\cX_\cA$ with a separable metric topology $\tau^{\Phi_*,\cA}$ by identifying  $x\in \cX_\cA$ with $(C(x), B(x), A(x), \Phi_*(C(x)), \Phi_*(B(x)))$ in the compact metric space $\cP(\bbN)^5$. 
\end{definition}

\begin{lemma}
	Using the notation from Definition~\ref{Def.PartitionXA}, $K^{\Phi_*,\cA,\cK}_0$ is a $\tau^{\Phi_*,\cA}$-open subset of $[\cX_\cA]^2$.  
\end{lemma}
\begin{proof}
		 Conditions \eqref{K0.1} and \eqref{K0.3} are clearly open. 
	The symmetric difference of $b'\cap a$ and $b\cap a'$ is included in $B_a\cap B_{a'}$, but since the family $\A_0$ is  tree-like, this is a finite set determined by the witness for (K1). [More precisely, if $m\in B_a\Delta B_{a'}$, then $B_a\cap B_{a'}$ is included in the finite set of points that are below $m$ in the tree ordering on $\bbN$ that witnesses $\cA$ is tree-like.] In other words, \eqref{K0.2} is open relative to \eqref{K0.1}, and this proves that the conjunction of all three conditions defines an open partition. 
\end{proof}

\begin{lemma}\label{Lemma.Jsigma.nonmeager} 
	Assume $\OCAT$, that $\cI$ is an ideal on $\bbN$, $\Phi\colon \cP(\bbN)\to \cP(\bbN)/\cI$ is a homomorphism. 
For every closed approximation $\cK$ to $\cI$, the hereditary set $\Jcont^{\cK^{16}}(\Phi) $
	has nonempty intersection with every perfect tree-like almost disjoint family.
\end{lemma}

\begin{proof} Let $\Phi_*$ be a lifting of $\Phi$. 
We will first prove that the hereditary set $\Jsigma^{\cK^4}\ucup \Jsigma^{\cK^4}$
	has nonempty intersection with every perfect tree-like almost disjoint family.
		Fix a perfect tree-like family $\cA$. 
For $n\in \bbN$ let 
\[
\cK_n=\cK\ucup [\bbN]^n. 
\]
Thus $\cI\subseteq \bigcup_n \cK_n$. 	Write $K^n_0=K^{\Phi_*,\cA,\cK_n}_0$ (see Definition~\ref{Def.PartitionXA}). When $\cX_\cA$ is endowed with the topology $\tau^{\Phi_*,\cA}$, this is an open partition.

\begin{claim}
	\label{C.K0-homogeneous} There are no uncountable $Z\subseteq \twoo$ and function $f\colon Z\to \cX_\cA$ such that $\{f(z), f(z')\}\in K^{\Delta(z,z')}_0$ for all distinct $z,z'$ in $Z$. 
\end{claim}

\begin{proof}{}Assume otherwise and fix $Z$ and $f$. By \eqref{K0.1} and \eqref{K0.2},  for all distinct $z$ and $z'$ in $Z$ we have $A(f(z))\neq A(f(z'))$ and \[
	B(f(z))\cap C(f(z'))=C(f(z))\cap B(f(z')).
\] 
	Let 
	$C=\bigcup_{z\in Z} C(f(z))$. Then $C\cap B(f(z))=C(f(z))$ for all $z\in Z$ and since $\Phi_*$ is a lifting of $\Phi$ we have 
	\[
	(\Phi_*(C)\cap\Phi_*(B(f(z))))\Delta \Phi_*(C(f(z)))\in \bigcup_n \cK_n.
	\] 
	Since $Z$ is uncountable, there is $n\in \bbN$ such that the set $Z'$ of all of $z\in Z$ satisfying 	$(\Phi_*(C)\cap\Phi_*(C(f(z))))\Delta \Phi_*(B(f(z)))\in \cK_n$ is uncountable.

This implies that all $z$ and $z'$ in $Z'$ satisfy 
\begin{align*}
	B(f(z))\cap C(f(z'))=^{\cK_n}B(f(z))\cap B(f(z'))\cap C =^{\cK_n} C(f(z))\cap B(f(z'))
\end{align*}
and therefore $B(f(z))\cap C(f(z'))=^{\cK_n^2} C(f(z))\cap B(f(z'))$. 
	Since $Z'$ is uncountable, there are $z$ and $z'$ in $Z'$ such that $\Delta(z,z')>n$. Therefore \eqref{K0.3} fails and  $\{f(z),f(z')\}\notin K_0^{\Delta(z,z')}$; contradiction.  
\end{proof}

Since $\OCAinfty$ is a consequence of $\OCAT$ (\cite[\S 5]{moore2021some}), there are sets $\cX_n$, for $n\in \bbN$, such that $\cX_\cA=\bigcup_n \cX_n$ and $[\cX_n]^2\subseteq K^n_1$ for all $n$.  Let $\calD_n\subseteq \cX_n$ be a countable $\tau^{\Phi_*,\cA}$ dense set. Since $\cA$ is uncountable, we can choose $\tilde A \in \cA\setminus \{A(x)\mid x\in \bigcup_n \calD_n\}$. 
If $m\geq 1$, $\bar x=(x_0,\dots x_{m-1})$ and $\bar y=(y_0,\dots y_{m-1})$ belong to $\cPN^m$,  and $k\in \bbN$, then we write\footnote{The relations $=^k$ for $k\in \bbN$ and $=^{\cK}$ for a closed hereditary $\cK$ should not be confused.}  
\[
\bar x=^k \bar y\text{ if and only if $x_i\cap k=y_i\cap k$ for all $i<m$.}
\]
For $m\in \bbN$ define 
\begin{align*}
m^+=&\min\{l>m\mid (\forall n\leq m) (\forall x\in \cX_n)A(x)=\tilde A\quad\Rightarrow \quad(\exists d\in \calD_n) B(d)\cap \tilde A\subseteq l\text { and}\\
&
(C(d), B(d), \Phi_*(C(d)), \Phi_*(B(d)))=^{m} 
(C(x), B(x), \Phi_*(C(x)), \Phi_*(B(x)))
\}. 
\end{align*}
Because $\calD_n$ is dense in $\X_n$ for every $n$, $\tilde A\neq D(x)$ for all $x\in \bigcup_n \calD_n$, and $\cP(m)^4$ is finite, 
 $m^+$ is finite for every $m$. 
Recursively define $m(j)$ for $j\in \bbN$ by 
\[
m(0)=0\text{ and }m(i+1)=m(i)^+\text{ for all $i$}.
\] 
This is a strictly increasing sequence, 
and  we let  
\[
B_0=\bigcup_i [m(2i), m(2i+1))\cap \tilde A, \qquad B_1=\tilde A\setminus B_0 
\]
so that $B_0\sqcup B_1=\tilde A$. 
We will prove that $B_j\in\Jsigma^{\cK^2}(\Phi)$ for $j=0,1$. 

For each $n$ let 
\begin{align*}
\cZ(n)=\{&(X,Y)\mid X\subseteq B_0, Y\subseteq \bbN, (\forall j\geq n) (\exists d\in \calD_n) B(d)\cap \tilde A\subseteq m(2j+2)\text{ and } \\
&(C(d), B(d), \Phi_*(C(d)), \Phi_*(B(d))=^{m(2j+1)} 
(X, B_0, Y, \Phi_*(B_0))\}.
\end{align*}

\begin{claim}\label{C.Zn}
	Suppose that $x\in \cX_n$ and $B(x)=B_0$. Then the following holds. 
	\begin{enumerate}
		\item $(C(x),\Phi_*(C(x)))\in \cZ(n)$. 
		\item If $(C(x),Y)\in \cZ(n)$ then $\Phi_*(B_0)\cap Y=^{\cK_n^2} \Phi_*(C(x))\cap \Phi_*(B_0)$. 
	\end{enumerate}
		\end{claim}

\begin{proof} If $x\in \cX_n$, $B(x)=B_0$,  and $m(2j+1)\geq n$  then since $\calD_n$ is $\tau^{\cA,\Phi_*}$-dense in $\cX_n$,  some $d\in \calD_n$ satisfies 	$B(d)\cap \tilde A \subseteq  m(2j+2)$ and 
	\[
	(C(d), B(d), \Phi_*(C(d)), \Phi_*(B(d))=^{m(2j+1)} 
	(C(x), B_0, \Phi_*(C(X)), \Phi_*(B_0). 
	\]  
  Since $j\geq n$ was arbitrary, $(C(x), \Phi_*(C(x))\in \cZ(n)$ follows. 

To prove the second part of the claim, towards contradiction suppose that $x$ is in~$\cX_n$, $B(x)=B_0$, $Y\subseteq \bbN$,  $(C(x),Y)\in \cZ(n)$, but 
\[
(\Phi_*(B_0)\cap Y) \Delta (\Phi_*(C(x))\cap \Phi_*(B_0))\notin {\cK_n^2} .
\]
Since $\cK_n$ is closed, there is $j\geq n$ large enough to have 
\begin{equation}\label{eq.notinKn}
(( \Phi_*(B_0)\cap Y) \Delta (\Phi_*(C(x))\cap \Phi_*(B_0)))\cap m(2j+1)\notin {\cK_n^2} .
\end{equation}
Since $(C(x), Y)\in \cZ_n$, some $d\in \calD_n$ satisfies
$B(d)\cap \tilde A\subseteq m(2j+2)$ and   \begin{equation}\label{eq.Y=Phi*(C(d))}
(C(d), B(d), \Phi_*(C(d)), \Phi_*(B(d)))=^{m(2j+1)} 
(C(x), B_0, Y, \Phi_*(B_0)).
\end{equation}
Since $B_0$ is disjoint from $[m(2j+1), m(2j+2))$, $B_0\cap C(d)=C(x)\cap B(d)$. 
As $\{x,d\}\in K^n_1$ and  $A(d)\neq A(x)$, we have 
\[
(\Phi_*(B_0)\cap \Phi_*(C(d)))\Delta (\Phi_*(C(x))\cap
\Phi_*(B(d)))\in \cK^2.
\]
Together with \eqref{eq.Y=Phi*(C(d))} this implies 
\[
((\Phi_*(B_0)\cap Y)\Delta (\Phi_*(C(x))\cap
\Phi_*(B_0)))\cap m(2j+1)\in \cK^2, 
\]
contradicting \eqref{eq.notinKn}. 
\end{proof}

We claim that each  $\cZ(n)$ is Borel. For a fixed $d\in\calD_n$  and $j\in \bbN$ the set  
\begin{align*}
\cZ(n,d,j)=
\{&(X,Y)\mid X\subseteq B_0, Y\subseteq \bbN, B(d)\cap \tilde A\subseteq m(2j+2)\text{ and }\\
&(C(d), B(d), \Phi_*(C(d)), \Phi_*(B(d)))=^{m(2j+1)} 
(X, B_0, Y, \Phi_*(B_0))\}
\end{align*}
is closed. Thus $\cZ(n)=\bigcap_j \bigcup_{d\in \calD_n} \cZ(n,d,j)$ is an $F_{\sigma\delta}$ set. By the Jankov, von Neumann theorem (\cite[18.A]{Ke:Classical}) there is a C-measurable function $\Theta_n$ whose domain includes  $\{C(x)\mid x\in \cX_n, B(x)=B_0\}$ such that  $(X,\Theta_n(X))\in \cZ(n)$ for all $X$ in the domain of $\Theta_n$.  
Let $\tilde \Theta_n(\cdot)=\Theta_n(\cdot)\cap \Phi_*(B_0)$.  
Since $\cX_\cA=\bigcup_n \cX_n$, Claim~\ref{C.Zn} implies that for every $X\subseteq B_0$ there is~$n$ such that $\Phi_*(X)=^{\cK_n^2} \tilde \Theta_n(X)$. Therefore $B_0$ belongs to $\Jsigma^{\cK^2}(\Phi)$, as required.  
By Lemma~\ref{L.K-baire-to-ctns}, $B_0$ is in $\Jsigma^{\cK^4}$. 

Analogous argument shows that $B_1\in \Jsigma^{\cK^4}$ and therefore $\tilde A\in \Jsigma^{\cK^4}\ucup \Jsigma^{\cK^4}$, as promised. By Lemma~\ref{L.JcontJ1JsigmaK}, $\tilde A\in \Jsigma^{\cK^8}$. 
 Since $\tilde A\in \cA\setminus \{A(d)\mid d\in \bigcup_n \calD_n\}$ was arbitrary, for every uncountable tree-like almost disjoint family $\cA$ all but countably many elements of $\cA$ belong to $\Jsigma^{\cK^8}$. 

We proceed to complete the proof of Lemma~\ref{Lemma.Jsigma.nonmeager}.  Fix a perfect tree-like almost disjoint family $\cB$. By Lemma~\ref{L.perfect.coarsening},  there is a perfect tree-like almost disjoint family~$\cA$ such that every $A\in \cA$ includes infinitely many elements of $\cB$. 
 By the first part of the proof,  all but countably many elements of $\cA$  belong to $\Jsigma^{\cK^8}$. 
 Every such element includes infinitely many disjoint elements of $\cB$, and by Proposition~\ref{P+.Jsigma-to-Jcont} all but finitely many of them belong to $\Jsigma^{\cK^{16}}$. 	Therefore $\Jsigma^{\cK^{16}}$  intersects every perfect tree-like almost disjoint family nontrivially, as required. 
 \end{proof}


\begin{proof}[Proof of Proposition~\ref{P.general.Jcont.nonmeager} for $F_\sigma$ ideals]
	Suppose $\Phi\colon \cP(\bbN)\to \cP(\bbN)/\cI$ is a homomorphism for an $F_\sigma$ ideal $\cI$. 
By Lemma~\ref{L.SolMa}, there is a closed approximation  $\cK$ to $\cI$ such that $\cI=\cK\ucup \Fin$. Also, $\cK^{16}\subseteq \cI$ hence $\cI=\cK^{16}\ucup \Fin$, thus  $\Phi$ and $\cK$ satisfy the assumptions of  Lemma~\ref{Lemma.Jsigma.nonmeager}.  Therefore the ideal $\Jcont^{\cK^{16}}(\Phi)$ includes $\Jcont(\Phi)$, and $\Jcont(\Phi)$ intersects every perfect  tree-like almost disjoint family, as required. 
\end{proof}

In order to extend the conclusion of Proposition~\ref{P.general.Jcont.nonmeager}  to countably determined ideals we need the following lemma.

\begin{lemma}\label{L.countable-approximation-lifting}
	Suppose that ideal $\cI$ is an intersection of a sequence of Borel sets, $\cI=\bigcap_n \cB_n$.  Then for every homomorphism $\Phi\colon \cP(\bbN)\to \cP(\bbN)/\cI$ the following are equivalent. 
	\begin{enumerate}
		\item $\Phi$ has a continuous lifting. 
		\item $\Phi$ has a Borel-measurable $\cB_n$-approximation $\Theta_n$ for every $n$. 
	\end{enumerate}
\end{lemma}

\begin{proof}Only the converse implication requires a proof. 
	Let 
	\[
	\cX=\{(a,b)\mid \Theta_n(a)\Delta b\in \cB_n\ucup \Fin\text{for all $n$}\}.
	\]
	Since all $\cB_n$ and $\Theta_n$ are Borel, so is $\cX$. By the 
	Jankov, von Neumann uniformization  
	theorem (\cite[18.A]{Ke:Classical})
	there is a
	C-measurable function $\Theta\colon \p(B)\to \p(\N)$ such that 
	$(a,\Theta(a))\in \cX$ for all $a$, and   therefore $\Theta$ is a
	lifting of a 
	homomorphism $\Phi$ on $\p(B)$. 
	By Corollary~\ref{C.Stabilization}, $\Phi$ has a continuous lifting. 
\end{proof}

\begin{proof}[Proof of Proposition~\ref{P.general.Jcont.nonmeager}, the general case] Fix a perfect tree-like almost disjoint family $\cA$.
Let $\cK_n$, for $n\in \bbN$, be closed approximations to $\cI$ such that $\cI=\bigcap_n (\cK_n^{16}\ucup \Fin)$. Lemma~\ref{Lemma.Jsigma.nonmeager} implies that $\Jcont^{\cK_n^{16}}$ intersects $\cA$ nontrivially for every $n$. This implies that $\Jcont^{\cK_n^{16}}$ contains all but countably many elements of $\cA$.   Then all but countably many elements of $\cA$ belong to $\bigcap_n \Jcont^{\cK_n^{16}}$. By Corollary~\ref{C.Stabilization}, each of these elements belongs to $\Jcont$.  
\end{proof}

\section{Uniformization modulo a countable ideal}\label{S.UniformizationFin}
The case when $\cI=\Fin$ of Theorem~\ref{T.OCAsharp-Fin} below was proven from $\OCAT$ and $\MA$ in \cite{Ve:OCA} and from $\OCAT$ in \cite[Theorem~3.3]{de2023trivial}, and the general case was proven from $\OCAT$ and $\MA$ in \cite[Theorem~1.9.2]{Fa:AQ}.  The proof for $\Fin$ given here is slightly shorter than the proof in \cite{de2023trivial}. It uses `Biba’s trick’, first used in the proof of \cite[Proposition~5.6]{de2023trivial}. A more sophisticated (read: more complicated) use of Biba’s trick can be seen in the proof of Proposition~\ref{P.uniformization}  .

\begin{theorem} \label{T.OCAsharp-Fin}
	Assume  $\OCAT$.  If $\cI$ is a countably generated ideal on $\bbN$ and 
	$\Phi\colon \p(\N)\to \p(\N)/\cI$
	is a  homomorphism, then $\Phi$ has a continuous lifting on the nonmeager ideal $\Jcont$.  
\end{theorem}



\begin{proof}[Proof of Theorem~\ref{T.OCAsharp-Fin}]  By Theorem~\ref{T.HNM}, if $\Jcont$ were meager then there would be a perfect almost disjoint family disjoint from $\Jcont$. 
	However,  $\Jcont$ intersects every perfect  tree-like almost disjoint family by  Proposition~\ref{P.general.Jcont.nonmeager} and is therefore nonmeager.  If $X=\{n\mid \{n\}\notin \cI\}$, then the homomorphism $A\mapsto \Phi(A)\cap X$ is already completely additive. We may therefore assume $\cI\supseteq \Fin$. 
		We will first prove the theorem in the case when $\cI=\Fin$. 
	
	Fix a homomorphism $\Phi\colon \cP(\bbN)\to \cP(\bbN)/\Fin$. 
	For each $A\in  \Jcont(\Phi)$ the restriction of $\Phi$ to $\cP(A)$ has a continuous lifting, and by \cite{Ve:Definable} (as extended in \cite[Theorem~1.6.1]{Fa:AQ}) has a completely additive lifting. This lifting is of the form 
	$B\mapsto h_A^{-1}(B)$ for a function  
	\[
	h_A\colon \Phi_*(A)\to A. 
	\]
	We fix these functions for a moment, to start with, but reserve all rights to modify then and re-evaluate the open partitions used in the proof as convenient.

	The first step will be to verify that the family $h_A$, for $A\in  \Jcont$, has an appropriate coherence property .

	\begin{claim}\label{L.hA.coherent}
		For all $A$ and $B$ in $\Jcont$ the set 
		\[
		\Diff(A,B)=\{n\in \Phi_*(A)\cap\Phi_*(B), h_A(n)\neq h_B(n)\}
		\]
		is finite. 
	\end{claim}

	\begin{proof}  Assume otherwise and fix $A$ and $B$ such that $C=\Diff(A,B)$ is infinite. Define $c\colon [C]^2\to \{0,1,2\}$ as follows. 
		\[
		c(\{m,n\}_<)=
		\begin{cases}
			0& \text { if } h_A(m)=h_B(n)\\
			1& \text { if } h_B(m)=h_A(n)\\	
			2& \text { if } h_A(m)\neq h_B(n)\text{ and } h_B(m)\neq h_A(n).
		\end{cases}
		\]
		Let $C_0\subseteq C$ be an infinite homogeneous set. If it is $2$-homogeneous then $A_0=h_A[C_0]$ and $B_0=h^B[C_0]$ are disjoint sets such that $\Phi_*(A_0)\cap \Phi_*(B_0)$ is infinite\footnote{Note that we are not assuming $\ker(\Phi)\supseteq \Fin$.} (hence $\Fin$-positive); contradiction. 
		
		If $C_0$ is $0$-homogeneous, then for all $m< m'<n$ in $C_0$ we have $h_B(m')=h_A(m)=h_B(n)=h_A(m')$, contradicting the choice of $C$. The case when $C_0$ is $1$-homogeneous leads to a contradiction by an analogous argument. 
	\end{proof}

	Extend each $h_A$ to a function $h_A^+$ from $\bbN$ into $\bbN_*=\bbN\cup \{\infty\}$ by $h_A^+(m)=\infty$ for $m\notin \Phi_*(A)$. 
	Identify $A\in \Jcont$ with the pair $(A,h_A^+)\in\Jcont\times {\bbN_*}^\bbN$. The right-hand side is a subspace of the Polish space $\cP(\bbN)\times {\bbN_*}^\bbN$ and we use this identification to equip $\Jcont$ with a separable metric topology $\tau$.

	For $t\Subset \bbN$ and $A,B$ in $\Jcont$ we say that $A$ and $B$ \emph{conflict on $t$} if $\Diff(A,B)\supseteq t$. For a fixed $t\Subset \bbN$, the set $U_t$ of pairs $\{A,B\}$ that conflict on $t$ is a symmetric $\tau$-open subset of $[\Jcont]^2$. 
	
This is a great moment to take a glance at the statement of $\OCAsharp$ (Definition~\ref{Def.OCAsharp}).
	For $m\geq 1$ let 
	\[
	\cV_m=\{U_s\mid s\in [\bbN]^{n}, \text{ with } n=m+(4^{m+1}-1)/3\}.
	\]  
	Since $t\subseteq t'$ implies $U_t\supseteq U_{t'}$, we have $\cV_m\supseteq \cV_{m+1}$ for all $m$.

	\begin{claim}\label{C.no-triple}
		There is no triple $(Z,f,\rho)$ with $Z\subseteq \twoo$ uncountable, $f\colon Z\to \Jcont^\cK$,  and $\rho\colon \Delta(Z)\to \bigcup_m \cV_m$ such that $\rho(s)\in \cV_{|s|}$ for all $s$ and 
		\(
		\{f(x),f(y)\}\in \rho(x\wedge y)
		\)
		for all distinct $x,y$ in $Z$. 
	\end{claim}
	\begin{proof}
		Assume otherwise and fix $Z$, $f$, and $\rho$.  We may assume that $Z$ has no isolated points, in which case $(\Delta(Z),\sqsubseteq)$ is a perfect tree. (It is not necessarily downwards closed in $\twolo$.) If $\rho(s)=U$ then $U=U_t$ for some $t$ of cardinality $|s|+(4^{|s|+1}-1)/3$; we write $A(s)=t$. A simple combinatorial argument  (\cite[Lemma~3.8]{de2023trivial})  implies that there are pairwise disjoint  $B(s)\subseteq A(s)\subseteq |s|$ such that $|B(s)|=2^{|s|}$ for all $s\in \Delta(Z)$.  Let $\cS_m$ denote the $m$-th level of $\Delta(Z)$ and note that $s\in \cS_m$  implies $|B(s)|\geq 2^m$. There are therefore disjoint sets $J_t$, for $t\in \{0,1\}^m$, such that $J_t\cap B(s)\neq \emptyset$ for all $s\in \cS_m$  and $\bigcup\{B(s)\mid s\in \cS_m\}=\bigcup\{J_t\mid t\in \{0,1\}^m\}$. For $g\in \twoo$ let 
		\[
		D(g)=\bigcup_n J_{g\rs n}.  
		\]
		Since the finite sets $J_t$ are nonempty and disjoint, $\{D(g)\}$ is a perfect tree-like almost disjoint family. Therefore $\OCAT$ and Proposition~\ref{P.general.Jcont.nonmeager} together imply that $D(g)\in \Jcont$ for some $g$. The salient property of $D(g)$ is that $D(g)\cap A(s)\neq \emptyset$ for all $s\in \Delta(Z)$. By Claim~\ref{L.hA.coherent}, for every $x\in Z$ there is $n(x)$ such that all $j\in (f(x)\cap D(g))\setminus n(x)$  satisfy $h_{D(g)}(j)=h_{f(x)}(j)$. 
		
		Fix $n$ such that $Z'=\{x\in Z\mid n(x)=n\}$ is uncountable. As the sets $A(s)$ are nonempty and disjoint, the   set $\{s\in \Delta(Z)\mid A(s)\cap n\neq \emptyset\}$ is finite.  Hence we can  choose distinct $x$ and $y$ in $Z'$ such that $A(x\wedge y)\cap n=\emptyset$.  
		Therefore $D(g)\cap A(x\wedge y)\neq \emptyset$, and $f(x)$ and $f(y)$ conflict on $D(g)\cap A(x\wedge y)\neq \emptyset$. However, each one of $h_{f(x)}$ and $h_{f(y)}$ agrees with $h_{D(g)}$ on this set; contradiction. 
	\end{proof}
	
	Since $\OCAsharp$ is a consequence of $\OCAT$ (\cite[Theorem~3.3]{de2023trivial}), by Claim~\ref{C.no-triple} and $\OCAsharp$ there are $\cX_n$, for $n\in \bbN$, such that $\Jcont=\bigcup_n \cX_n$ and $[\cX_n]^2\cap \cV_n=\emptyset$ for all $n$. 
	Since $\Jcont(\Phi)$ is nonmeager,  we can fix $n$ such that 
	$\cX_n$ is nonmeager. 
	
	Next, we attempt to recursively choose an increasing sequence $n_i$ and $k_i\neq l_i$ for $i\in \bbN$ such that the following holds for all $m$.\footnote{The remaining  part of the proof  is Biba's trick.} 
	\begin{enumerate}
		\item The set $\cF_{0,m}=\{A\in \cX_n\mid h_A(n_i)=k_i$ for all $i<m\}$ is nonmeager. 
		\item The set $\cF_{1,m}=\{B\in \cX_n\mid h_B(n_i)=l_i$ for all $i<m\}$ is nonmeager.
	\end{enumerate} 
	Since $[\cX_n]^2\cap \bigcup\cV_n=\emptyset$, a recursive construction of such sequences has to stop at a finite stage (more precisely, before the $n+(4^{n+1}-1)/3$-th stage). We therefore have $m$ (possibly $m=0$, with $n_{-1}=0$), $n_i,k_i,l_i$, for $i<m$ such that for all $n>n_{m-1}$ and all $k\neq l$ at least one of the sets
	\[
	\{A\in \cF_{0,m}\mid h_A(n)=k \}
	\text{
		or
	}
	\{B\in \cF_{1,m}\mid h_B(n)=l \}
	\]
	is meager. 
	
	Let $D$ be the set of $n>n_{m-1}$ 
	such that both $\cF_0=\{A\in \cF_{0,m}\mid n\in \Phi_*(A)\}$ and $\cF_1=\{B\in \cF_{1,m}\mid n\in \Phi_*(B)\}$ are nonmeager. Fix $n\in D$. If $k$ and $l$ are such that $\{A\in \cF_0\mid h_A(n)=k\}$ and $\{B\in \cF_1\mid h_B(n)=l\}$ are nonmeager, then we have $k=l$ (otherwise we would have set $n_m=n$, $k_m=k$, and $l_m=l$). Therefore there is $k$ such that $h_A(n)=k=h_B(n)$ for a relatively  comeager (in $\cF_0$) set of $A\in \cF_0$ and a relatively  comeager (in $\cF_1$) set of $B\in \cF_1$. 
	
	Define $h\colon D\to \bbN$ by letting $h(n)=k$,  for $k$ as in the previous line. 
	
		\begin{claim}\label{L.hA.coherent.plus}
		For all $A$ in $\Jcont$ the set 
		\[
		\Diff(A,h)=\{n\in \Phi_*(A)\cap D\mid h_A(n)\neq h(n)\}
		\]
		is finite. 
	\end{claim}

	\begin{proof}  The proof is similar to the proof of Claim~\ref{L.hA.coherent}. Assume otherwise and fix~$A$ such that $C=\Diff(A,h)$ is infinite. Defining $c\colon [D]^2\to \{0,1,2\}$ as in the proof of Claim~\ref{L.hA.coherent}, we obtain an infinite $D_0\subseteq D$ such that $h_A[D_0]$ and $B_0=h[D_0]$ are disjoint. If $B_0$ is  infinite, then by nonmeagerness of $\cX_n$ there is $C\in \cX_n$ such that $C\cap B_0$ is  infinite.  Then $h_{C} $ and $h_A$ disagree on the infinite set $h^{-1}[C\cap B_0]$, contradicting Claim~\ref{L.hA.coherent}. 
		
		Since we are not assuming that $\ker(\Phi)\supseteq \Fin$, there is a possibility that $B_0$ is finite. Fix $j\in B_0$ such that $h^{-1}[\{j\}]$ is infinite and fix $C\in \cX_n$ such that $j\in C$. Then $\Phi(\{j\})=[D_0]_{\Fin}$, but $D_0$ is also equal (modulo finite) to the $\Phi_*$-image of a subset of $A$ disjoint from $\{j\}$; contradiction. 
	\end{proof}

	\begin{claim}\label{C.AisFinite} The set $A=\bbN\setminus h[D]$ is finite. 
	\end{claim}
	\begin{proof}
		Assume otherwise. Since $\cF_1$ is nonmeager, there is $B\in \cF_1$ such that $A\cap B$ is infinite. The function $h_B\colon \Phi_*(B)\to B$ has cofinite range, therefore $h_B(j)\in A$ for some $j>n$; contradiction. 
	\end{proof}
	
	If $A\in \ker(\Phi)$, then  $B\mapsto h^{-1}(B)$ is a lifting of $\Phi$ on $\Jcont$. We need to consider the case when $A\notin \ker(\Phi)$.  Since $A$ is finite by Claim~\ref{C.AisFinite}, it is straightforward to extend $h$ to  $\Phi_*(A)$ so that $h^{-1}(\{j\})=^*\Phi_*(\{j\})$ for all $j\in A$. Then $B\mapsto h^{-1}(B)$ is a lifting of $\Phi$ on $\Jcont$, as required. 
	We have proven that the function $X\mapsto h^{-1}(X)$ is a continuous lifting of $\Phi$ on the nonmeager ideal $\Jcont$. This completes the proof in the case when $\cI=\Fin$.

	Now consider the general case, when $\cI$ is generated by an increasing sequence of subsets of $\bbN$, $A_n$, for $n\in \bbN$. If there is $m$ such that $A_{n+1}\setminus A_n$ is finite for all $n\geq m$, then the range of  $\Phi$ can be identified with $\cP(\bbN\setminus A_m)/\Fin$ and the conclusion follows from the first part of the proof. 
	We may therefore assume that $A_{n+1}\setminus A_n$ is infinite for infinitely many $n$. By passing to a subsequence, we may assume $A_{n+1}\setminus A_n$ is infinite for all $n$. Therefore $\cI$ is isomorphic to the ideal 
	\[
	\FinO=\{B\subseteq \bbN^2\mid(\exists m) B\subseteq m\times \bbN\}.
	\]
	The remaining part of the proof borrows some of the ideas from the proof of \cite[Theorem~1.9.2]{Fa:AQ}.  It will be convenient to use Greek letters for the elements of~$\bbNN$. For $\alpha$ and $\beta$ in $\bbNN$ we write $\alpha\leq^* \beta$ if $(\forall^\infty j)\alpha(j)\leq \beta(j)$. 
	For $\alpha\in \bbNN$ let $\Gamma_\alpha=\{(m,n)\mid n<\alpha(m)\}$. Then $\cI\cap \cP(\Gamma_\alpha)$ is the ideal of finite subsets of $\Gamma_\alpha$.  Therefore, by the first part of the proof there are $A_\alpha\subseteq \Gamma_\alpha$, $h_\alpha\colon A_\alpha\to \bbN$,  and an ideal $\cJ_\alpha$ on~$\bbN$ such  that $\cJ_\alpha$ intersects every perfect tree-like almost disjoint family nontrivially and the function
	\[
	\Psi_\alpha(X)= h_\alpha^{-1}(X)
	\]
	satisfies (with $\Phi_*$ denoting a lifting of $\Phi$ and abusing notation to denote the ideal of finite subsets of $\bbN^2$ by $\Fin$)  
	\[
	(\Psi_\alpha(X)\Delta \Phi_*(X))\cap \Gamma_\alpha\in \Fin
	\]
	for all $X\in \cJ_\alpha$.

	Define a partition of $[\bbNN]^2=K_0\cup K_1$ by 	setting $\{\alpha,\beta\}\in  K_0$ if and only if 
	\[
h_\alpha((i,j))\neq h_\beta((i,j))\text{ for some $i$ and $j<\min(\alpha(i),\beta(i))$}. 
	\]
	We identify $\alpha$ with $(\alpha,h_\alpha)$  and $\bbNN$ with a subset of the Polish space $\bbNN\times \bbN^{\bbN^2}$, and use this identification to topologize $\bbNN$.  Then $K_0$ is symmetric and open in this topology.

	\begin{claim} \label{C.FinO} There is no uncountable $K_0$-homogeneous $\cX\subseteq \bbNN$ . 
	\end{claim} 
\begin{proof}
	Assume otherwise.  $\OCAT$ implies that every subset of $\bbN$ of cardinality~$\aleph_1$ is $\leq^*$-bounded (\cite{To:Partition}).  By this fact and a little counting argument,  we may assume that there are an uncountable $K_0$-homogeneous set $\cX$ and $\alpha\in \bbNN$ such that $\beta\leq\alpha$ for all $\beta\in \cX$. 
	Therefore $h_\alpha\rs \Gamma_{\beta}=^*h_{\beta}$ and there is $m=m(\beta)$ such that $h_\alpha((i,j))=h_{\beta}(i,j)$ for all $i\geq m$ and $j<\min(\alpha(i), \beta(i))$. 
	Let $m$ be such that the set $Z'=\{x\in Z\mid m(x)=m\}$ is uncountable, and choose $\beta$ and $\gamma$ in $Z'$ such that $h_{\beta}((i,j))=f_{\gamma}(i,j)$ for all $i<m$ (this is possible since there are only finitely many possibilities). Then $\{\beta,\gamma\}\in K_1^m$; contradiction. 
\end{proof}	
	
	By Claim~\ref{C.FinO}, $\OCAT$ implies that $\bbNN=\bigcup\cX_n$ and $[\cX_n]^2\subseteq K_1$ for all $n$. By \cite{Ku:Box}, there are $n$ and $m$ such that for every $\alpha\in \bbNN$ some $\beta\in \cX_n$ satisfies $\alpha(i)\leq \beta(i)$ for all $i\geq m$. By $K_1$-homogeneity, $h=\bigcup_{\beta\in \cX_n} h_\beta$ is a function. Its domain includes $[m,\infty)\times \bbN$.

	We claim that $\{A\subseteq \bbN\mid \Phi_*(A)\Delta h^{-1}(A)\in \FinO\}$ includes an ideal that intersects every perfect tree-like almost disjoint family. 
	Fix a perfect tree-like almost disjoint family $\cA\{J_s\}$, for a family of disjoint finite sets $J_s$, for $s\in \twolo$. Since $\FinO$ is an $F_\sigma$ ideal, by Proposition~\ref{P.general.Jcont.nonmeager} there is $f\in \twoo$ such that $A(f)=\bigcup_n J_{f\rs n}$ belongs to $ \Jcont$. Let $\Theta\colon A(f)\to \cP(\bbN^2)$ be a continuous lifting of $\Phi$ on $\cP(A(f))$. 
	We will use the following characterization of $\FinO$: Some  $C\subseteq \bbN^2$ belongs to $\FinO$ if and only if $C\cap \Gamma_\alpha$ is finite for all $\alpha\in \bbNN$. 
	By Lemma~\ref{L.Xstabilizer}, for every $\alpha\in \bbNN$ and for every $B\subseteq A(f)$ the set $\Theta(B)\Delta h_\alpha^{-1}(B)\cap \Gamma_\alpha$ is finite. Therefore $\Theta(B)\Delta h^{-1}(B)\cap \Gamma_\alpha$ is finite for all $\alpha$, and $\Theta(B)\Delta h^{-1}(B)\in \FinO$. 
	
	Thus $B\mapsto h^{-1}(B)$ is a lifting of $\Phi$ on an ideal that intersects an arbitrary perfect tree-like almost disjoint family, as required. 
		\end{proof}

\begin{proof}[Proof of Theorem~\ref{T.OCAonly}] Fix an isomorphism $\Phi\colon \cPN/\cI'\to \cPN/\cI$ for an analytic ideal $\cI'$ and a countably generated ideal $\cI$ on $\bbN$. 
	By $\OCAT$ and Theorem~\ref{T.OCAsharp-Fin} there is a continuous function $\Theta\colon \cPN\to \cPN$ that lifts $\Phi$ on a nonmeager ideal $\cJ$. By the methods of \cite[\S 1]{Fa:AQ}, there are $A\subseteq \bbN$ and $h\colon A\to \bbN$ such that $X\mapsto h^{-1}(X)$ is a lifting of $\Phi$ on $\cJ$. Since $\Phi$ is an isomorphism and since $\cI'$ is universally Baire, we have $\bbN\setminus A\in \cI$. Since $\Phi$ is an isomorphism, by removing a set in $\cI$ from $A$ we can assure that $h$ is a biiection, and therefore $\cI'$ is Rudin--Keisler isomorphic to $\cI$. This also shows that all automorphisms of $\cPN/\cI$ are trivial. 
\end{proof}
\section{Uniformization modulo $\cI$}

The main result of this section is Proposition~\ref{P.uniformization}  which together with Proposition~\ref{P.general.Jcont.nonmeager} completes the proof of Theorem~\ref{T.OCA-l}. 
Lemma~\ref{L.Roberts} below is based on a lemma due to J.W. Roberts that was one of the (simpler) ideas involved in the solution to Maharam's problem (\cite{Tal:Maharam}). 

\begin{lemma}
	\label{L.Roberts}
	For every $m\in \bbN$, if for  all $s\subseteq m$ there are  $l(s)\geq 2^{2m}$  and an $l(s)$-tuple $\vec n(s)$: $n_0(s)<n_1(s)<\dots < n_{l(s)}(s)$ in $\bbN$, the there are pairwise disjoint sets $A(t)\subseteq \bbN$, for $t\subseteq m$, such that for all $s$ and $t$ some $i=i(s,t)$ satisfies $[n_i(s),n_{i+1}(s))\subseteq A(t)$. 
\end{lemma}

\begin{proof}
	Enumerate $2^m$ as $s(i)$, for $i<2^m$, by using the following algorithm. 
	Choose $s(0)$ so that $n_{2^m}(0)$ is minimal possible. If $s(i)$ for $i<k$ had been chosen, then let $s(k)$ be such that $n_{(k+1)2^m}(s)$ is minimal possible among $\{s<2^m\mid s\neq s(i)$ for $i<k\}$. This describes the construction. For all $k<2^m$ we have $n_{k 2^m}(s(k))\geq n_{k 2^m} (s(k-1))$ hence the intervals $J_{k}=[n_{k2^m}(s(k)), n_{(k+1) 2^m}(s(k)))$, for $k<2^m$,  are  disjoint and the sequence $n_{(k+1) 2^m}(s(k))$, for $k<2^m$,  is nondecreasing. 
	
	Fix a bijection. $g\colon 2^m\to \cP(m)$ and for $t\subseteq m$ let 
	\[
	A(t)=\bigcup_{k<2^m} [n_{k 2^m+g(t)}(s(k)), n_{k 2^m +g(t)+1} (s(k))). 
	\]
	The sets $A(t)$ are disjoint, and for all $s=s(k)$ and $t$ we have 
	\[
	A(t)\supseteq [n_{k 2^m+g(t)}(s(k)), n_{k 2^m +g(t)+1} (s(k)))	,
	\] 
	as required. 
\end{proof}

The fact that Roberts’s lemma cannot be reasonaby extended to infinite families of finite sequences is one of the reasons why the construction of a pathological Maharam submeasure given in \cite{Tal:Maharam} is deep and beautiful. In our situation, we can get away with a little use of $\MAsigma$ (since the existence of a pathological Maharam submeasure is obviously  a $\Sigma^1_2$ statement,\footnote{It is considerably less obvious that it is a $\Delta_0$ statement, but this follows from Talagrand’s result.} $\MAsigma$ would be of no help there).

If $\cK$ is a closed hereditary set, $F$, $G$ are functions whose domains are subsets of $\cP(\bbN)$, and $A\subseteq \bbN$ is such that $\cP(A)\subseteq \dom(F)\cap \dom(G)$, then we write 
\begin{equation}\label{eq.FrsP(A)}
F\rs \cP(A)=^\cK G\rs \cP(A)
\end{equation}
if all $B\subseteq A$ satisfy $F(B)\Delta G(B)\in \cK$
and 
$F\rs \cP(A)=^{\cK\ucup \Fin} G\rs \cP(A)$ if all $B\subseteq A$ satisfy $F(B)\Delta G(B)\in \cK\ucup \Fin$.

\begin{lemma}\label{L.approximation.nonmeager}
	Assume that $\Phi\colon \cPN\to \cPN/\cI$ is a homomorphism, $\cK$ is a closed approximation to $\cI$,  and that $\cY_n$, for $n\in \bbN$, are hereditary sets such that $\bigcup_n \cY_n$ is nonmeager and there is a C-measurable $\cK$-approximation $\Theta_n$ to $\Phi$ on $\cY_n\ucup \cY_n$ for all $n$. 
	Then there is a continuous  $\cK^4$-approximation to $\Phi$ on  a relatively comeager hereditary subset of $\bigcup_n \cY_n$. 
\end{lemma}
\begin{proof} Our first task is to prove that we may assume each $\cY_n$ is closed under finite changes. Fix $n$ such that $\cY_n$ is nonmeager. By Lemma~\ref{L.HNM.1}, there is $k(n)$ such that for every $s\Subset [k(n),\infty)]$ the set $\tilde \cY_n=\{A\subseteq \bbN \mid s\cup A\in \cY_n\}$ is hereditary and nonmeager. 
	
	Since we are not assuming that $\ker(\Phi)$ includes $\Fin$, there is some extra work to do. Fix a lifting $F_n\colon \cP(k(n))\to \cP(\bbN)$ of the restriction of $\Phi$  to $\cP(k(n))$. Then define $\tilde \Theta_n\colon \tilde \cY_n\to \cP(\bbN)$ by 
	\[
	\tilde \Theta_n(A)=F_n(A\cap k(n))\cup \Theta_n(A\setminus k(n)). 
	\]
	Since $\cP(k(n))$ is finite, this function is C-measurable and it is a lifting of $\Phi$ on $\tilde \cY_n$.
	
	Let $X=\{n\mid \tilde \cY_n$ is nonmeager$\}$. Since $\bigcup_n \cY_n$ is nonmeager, $X$ is nonempty. 
	
	By Theorem~\ref{T.HNM}, $\calH=\bigcap_{n\in X} \tilde\cY_n$ is nonmeager, and it is clearly closed under finite changes of its elements. 
	
	An argument similar to Corollary~\ref{C.HereditaryComeager} and Corollary~\ref{C.Stabilization}   follows. 
	Let $\cG\subseteq \cPN$ be a dense $G_\delta$ set such that $\tilde \Theta_n\rs G$ is continuous for all $n\in X$ and $G\cap \tilde \cY_n=\emptyset$ for all $n\notin X$. 
	
	Choose disjoint $J_i\Subset\bbN$ and $t_i\subseteq J_i$ such that $\{A\mid (\exists^\infty i)A\cap J_i=t_i\}\subseteq G$. Since $\calH$ is hereditary and nonmeager, there is an infinite set $Y\subseteq \bbN$ such that $\bigcup_{i\in Y} J_i\in \calH$. Let $Y=Y_0\sqcup Y_1$ be a partition into two infinite sets, let $C_0=\bigcup_{i\in Y_0} t_i$, $C_1=\bigcup_{i\in Y_1}  t_i$, $B_0=\bigcup_{i\in Y_0} J_i$, and $B_1=\bbN\setminus B_0$. 
	Then $A\cap B_j=C_j$  for $j=0$ or for $j=1$ implies $A\in \calH$.  If in addition $A\in \cY_n$ for some $n$, then $(A\setminus B_j) \cup (A\cap B_j)\in \cY_n\ucup \cY_n$ for $j=0$ and for $j=1$. 
For $n\in X$ let (using a lifting $\Phi_*$ of $\Phi$)
\[
\Upsilon_n(A)=(\Phi_*(B_1)\cap \Theta_n((A\setminus B_0) \cup (A\cap B_0)))\cup 
(\Phi_*(B_0)\cap \Theta_n((A\setminus B_1) \cup (A\cap B_1))).
\]
Since the arguments of $\Theta_n$ in the definition belong to $G$, $\Upsilon_n$ is continuous. Since~$\Theta_n$ is a $\cK$-approximation to $\Phi$ on $\cY_n\ucup \cY_n$, $\Upsilon_n$ is a $\cK^2$-approximation to $\Phi$ on  a relatively comeager subset of $\cY_n$. 

Fix $m\in X$  and let $\Upsilon=\Upsilon_m$. We claim that $\Upsilon$ is a $\cK^4$-approximation to $\Phi$ on~$\cY_n$ for every $n\in X$. 
This follows from Lemma~\ref{P.Theta0Theta1} and completes the proof. 
\end{proof}

\begin{proposition} \label{P.uniformization} Assume $\OCAT$ and $\MAsigma$. Suppose that $\cI$ is an ideal on~$\bbN$ with closed  approximation $\cK$, $\Phi\colon \cP(\bbN)\to \cP(\bbN)/\cI$ is a homomorphism,
 and the hereditary set $\Jcont^\cK(\Phi)$ intersects every uncountable  tree-like almost disjoint family. 
		Then $\Phi$ has a continuous $\cK^{80}$-approximation  on  $\Jcont^\cK(\Phi)$. 
\end{proposition}

\begin{proof}
	For $A$ and $B$ in $\Jcont^\cK(\Phi)$ and $D\subseteq \bbN$ we say that $F^A$ and $F^B$ \emph{conflict} on~$D$ if there is $s\subseteq A\cap B$ such that 
	\[
	F^A(s)\Delta F^B(s)\cap D \notin \cK^{20}.
	\] 
	Since all $F^A$ and $F^B$ are continuous, $A$ and $B$ conflict on $D$ if and only if there is $s\Subset A\cap B$ such that $F^A(s)\Delta F^B(s)\cap D \notin \cK^{20}$.
	Identify $A\in \Jcont^\cK$ with the pair $(A,\{(s,F^A(s))\mid s\Subset A\})$. 
	This identifies $\Jcont^\cK$ with a subset of $\cP(\bbN)\times \cP(\bbN)^{\Fin}$, and endows  $\Jcont^{\cK}$ with the subspace topology, denoted  $\tau$. 
	
	\begin{claim} \label{C.conflict} For every $D\subseteq \bbN$, the set of all $\{A,B\}$ that conflict on $D$ is symmetric and  $\tau$-open. 
	\end{claim}
	
	\begin{proof}  Symmetry is obvious from the definition.  To prove that this set is open fix $A$ and $B$ that conflict on $D$.  Since $\cK^{20}$ is closed, we can fix open subsets $U_0$ and $U_1$ of $\cP(\bbN)$ such that $F^A(s)\cap D\in U_0$, $F^B(s)\cap D\in U_1$, and for all $X_0\in U_0$ and $X_1\in U_1$ we have $X_0\Delta X_1\notin \cK^{20}$.  Then the  sets $V_j=\{C\mid s\Subset C, F^C(s)\in U_j\}$  for $j=0,1$ are $\tau$-open neighbourhoods of $A$ and $B$ and $A’$ and $B’$ conflict on $D$ for all $A’\in U_0$ and $B’\in V_1$.
	\end{proof}
	
	By Lemma~\ref{L.Xstabilizer}, for all $A$ and $B$ in $\Jcont^\cK$ there is $k=k(A,B)\in \bbN$ such that for all $s\subseteq (A\cap B)\setminus k$ we have $(F^A(s)\Delta F^B(s))\setminus k\in \cK^{10}$. 
	For $m\in \bbN$, let $\cV_m$ be the set of all $U\subseteq [\Jcont^\cK]^2$ such that there are $m\leq n_0^U<n_1^U<\dots <n_{2^{2m}}^U$ for which $U$ is the set of all pairs $\{A.B\}\in [\Jcont^\cK]$ which conflict on $[n_i^U,n_{i+1}^U)$ for all $i<2^{2m}$. 
	By Claim~\ref{C.conflict}, each $\cV_m$ is a union of symmetric open subsets of $[\Jcont^\cK]^2$ and clearly $\cV_m\supseteq \cV_{m+1}$, hence these sets are as in the statement of $\OCAsharp$. 
	
	Next we verify that one of the alternatives of $\OCAsharp$ cannot hold. 
	
	\begin{claim}\label{C.no-triple.1}
		There is no triple $(Z,f,\rho)$ such that $f\colon Z\to \Jcont^\cK$, $Z\subseteq \twoo$ is uncountable,   and $\rho\colon \Delta(Z)\to \bigcup_m \cV_m$
 such that $\rho(s)\in \cV_{|s|}$ for all $s$ and 
		\[
		\{f(x),f(y)\}\in \rho(x\wedge y)
		\]
		for all distinct $x,y$ in $Z$. 
	\end{claim}
	\begin{proof}
		Assume otherwise and fix $Z$, $f$, and $\rho$. For $s\in \Delta(Z)$ let 
		\[
		I(s)=[|s|, n_{2^{2|s|}} ^{\rho(s)}). 
		\]
		By $\MAsigma$ and Lemma~\ref{L.MAsl.thin} there are an uncountable $Z'\subseteq Z$ and an increasing sequence $\{k_i\}$ such that for every $s\in \Delta(Z')$ some $m(s)$  satisfies 
		\[
		I(s)\subseteq [k_{m(s)}, k_{m(s)+1})
		\] 
		and $\cS_m=\{s\in \Delta(Z')\mid m(s)=m\}$ is the $m$-th level of the tree $(\Delta(Z'),\sqsubseteq)$. 
		
		Fix $m$. For each $s\in \cS_m$ we have $|s|\leq m$ and therefore $\rho(s)$ is given by $m\leq n_0^s<n_1^s<\dots <n_{l(s)}^s$ for $l(s)\geq 2^{2m}$. By Lemma~\ref{L.Roberts} there are disjoint $A(t)\subseteq [k_m,k_{m+1})$, for $t\in \{0,1\}^m$, such that for all $s\in \cS_m$ and $ t\in \{0,1\}^m$,  for some $i$ we have $A(t)\supseteq [n_i^s, n_{i+1}^s)$. The sets 
		\[
		A(h)=\bigcup_m A(h\rs m)
		\]
		for $h\in \twoo$, satisfy  $A(h)\cap A(h')\subseteq k_{\Delta(h,h')+1}$ for all distinct $h$ and $h'$. Therefore $A(h)$, for $h\in \twoo$, is an uncountable  tree-like almost disjoint family. By the assumption, $A(h)\in \Jcont^\cK$  for some $h$. By Lemma~\ref{L.Xstabilizer}, for every $A\in \Jcont^\cK$ there is $k(A)$  such that for all $s\subseteq (A\cap A(h))\setminus k$ we have $(F^A(s)\Delta F^{A(h)}(s))\setminus k\in \cK^{10}$. 
		Let $k$ be such that $Z''=\{z\in Z'\mid k(f(z))=k\}$ is uncountable. Choose $x$ and $y$ in $Z''$ such that $\Delta(x,y)>k$. Then for all $s\subseteq (A(h)\cap f(x)\cap f(y))\setminus k$  we have (see \eqref{eq.FrsP(A)})
		\[
		F^{f(x)}(s)\setminus k=^{\cK^{10}} F^{A(h)}(s)\setminus k =^{\cK^{10}} F^{f(y)} (s)\setminus k . 
		\]
		However, $A(h)\cap \rho(x\wedge y)$ includes an interval of the form $[n_i^s, n_{i+1}^s)$, for $s=\rho(x\wedge y)$ on which $F^{f(x)}$ and $F^{f(y)}$  conflict.  The minimum of this interval is at least $|s|\geq k$; contradiction. 
	\end{proof}

	Since $\OCAT$ implies $\OCAsharp$, by Claim~\ref{C.no-triple.1} we conclude that there are sets $\cX_n$, for $n\in \bbN$, such that  $\Jcont^\cK=\bigcup_n \cX_n$ and $[\cX_n]^2\cap \cV_n=\emptyset$  for all $n$. 

	We will use a version of Biba's trick similar to the one in  the final part of the proof of Theorem~\ref{T.OCAsharp-Fin}. 
			Let $n$ be such that $\cX_n$ is nonmeager in the original topology on $\cP(\bbN)$. We attempt to recursively choose an increasing sequence $n_i$, as well as $s_i\subseteq [n_{i-1},n_{i})$ (with $n_{-1}=0$), $u_i$, and $v_i$  for $i\in \bbN$ such that the following holds for all $m$.
	\begin{enumerate}
		\item The set $\cF_{0,m}=\{A\in \cX_n\mid F^A(s_i)=u_i$ for all $i<m\}$ is nonmeager. 
		\item The set $\cF_{1,m}=\{B\in \cX_n\mid F^B(s_i)=v_i$ for all $i<m\}$ is nonmeager.
		\item $u_i\Delta v_i\notin \cK^{20}$ for all $i<m$. 
	\end{enumerate} 
	Since $[\cX_n]^2\cap \bigcup\cV_n=\emptyset$, a recursive construction of such sequences has to stop at a finite stage (more precisely, before the $2^{2n}$-th stage). We therefore have $m$ (possibly $m=0$), $n_i,s_i,u_i, v_i$, for $i<m$ such that for all $s\subseteq [n_{m-1},\infty)$ the set 
	\[
	\{(A,B)\in \cF_{0,m}\times \cF_{1,m}\mid s\subseteq A\cap B, F^A(s)\Delta F^B(s)\notin \cK_n^{20}\}
	\]
	is meager.  By increasing $n_{m-1}$ if needed, we can assure that for every $s\subseteq [n_{m-1}, \infty)$ both sets $\{A\in \cF_{0,m}\mid s\subseteq A\}$ and $\{B\in \cF_{1,m}\mid s\subseteq B\}$ are nonmeager.

	By Lemma~\ref{L.HNM.1}, for a  large enough $k\geq n$, for every interval of the form $[k,l)$ for $l>k$ there is an $A(l)\in \cX_n$ such that $[k,l)\subseteq A(l)$. Identify~$\cX_n$ with a subset of $\cP(\bbN)\times \cP(\bbN)^{\Fin}$ as in the definition of the topology $\tau$  (see the paragraph preceding Claim~\ref{C.conflict}). 
Let 
\[
\cY=\{(X,Y)\in \cPN^2\mid(\forall j)(\forall^\infty l) \min(Y\Delta F^{A(l)} (X\cap [k,l)))\cap j\in \cK_n^{20} \}.
\]
This is a Borel (more precisely, $F_{\sigma\delta}$) set. We claim that the section $\cY_X$ is nonempty for every $X\in \cX_n$. Fix $X$. Since $\cP(j)$ is finite, for every $j$ the set 
\[
T_{X,j}=\{s\subseteq j\mid (\forall^\infty l) F^{A(l)} (X\cap [k,l))\cap j=s\}
\] 
is nonempty. Then $T_X=\bigcup_j T_{X,j}$ is a finitely branching tree with respect to the end-extension,  and $T_{X,j}$ is its $j$-th level. Since each $T_{X,j}$ is nonempty, $T_X$ has an infinite branch $Y$. Clearly $Y$  belongs to $\cY_X$. 

Because $[\cX_n]\cap \cV_n=\emptyset$ and because $k\geq n$, for all $(X,Y)\in \cY$ such that $X\in \hat \cX_n$ and $A\in \cX_n$ satisfies $X\subseteq A$ we have $F^A(X)\Delta Y\in \cK_n^{20}$. 

The set $\cZ=\{X\mid \cY_X\neq \emptyset\}$ is analytic. Since it is a continuous image of $\cY$, 
by the Jankov, von Neumann uniformization theorem (\cite[18.A]{Ke:Classical}) there is a C-measurable selection $\Theta_n\colon \cZ\to \cPN$ for $\cY$. By the previous paragraph, for every $X\in \cZ\cap \hat \cX_n$ and $A\in \cX_n$ such that $X\subseteq A$ we have that $\Theta_n(X)\Delta F^A(X)\in \cK_n^{20}$. 

By Lemma~\ref{L.approximation.nonmeager} applied to $\cK_n^{20}$, there is a continuous $\cK_n^{80}$-approximation to $\Phi$ on a relatively comeager, hereditary subset of $ \Jcont$. 
\end{proof}

\section{Proofs of Theorem~\ref{T.OCA-l} and Theorem~\ref{T.Main}}
\label{S.Proofs}

\begin{proof}[Proof of Theorem~\ref{T.OCA-l}] Suppose that $\cI$ is countably 80-determined by closed approximations $\cK_n$, for $n\in \bbN$ and that $\Phi\colon \cPN\to \cPN/\cI$ is a homomorphism. 
	By Lemma~\ref{Lemma.Jsigma.nonmeager}, $\Jcont^{\cK_n^{16}}(\Phi) $
	has nonempty intersection with every perfect tree-like almost disjoint family.
Therefore each one of these ideals is nonmeager.

By Proposition~\ref{P.uniformization}, $\Phi$ has a continuous $\cK_n^{80}$ approximation on $\Jcont$. 
Since this is true for all $n$, Lemma~\ref{L.countable-approximation-lifting} applied with $\cB_n=\cK_n^{80}$ implies that $\Phi$ has a continuous lifting on a relatively comeager subset $\cX$ of $\Jcont$. 
Finally, by Corollary~\ref{C.Stabilization} there are a partition $\bbN=A_0\sqcup A_1$ and sets $C_0\subseteq A_0$ and $C_1\subseteq A_1$ such that for every $X\in \Jcont$ both $(X\cap A_0)\cup C_1$ and $(X\cap A_1)\cup C_0$ belong to $\Jcont^2\cap \cX$. But $\Jcont$ is an ideal, hence $\Jcont^2=\Jcont$, and the argument from Corollary~\ref{C.Stabilization} gives a continuous lifting of $\Phi$ on the nonmeager ideal $\Jcont$.  
\end{proof}

\begin{proof}[Proof of Theorem~\ref{T.Main}] Suppose that $\Phi\colon \cPN/\cI'\to \cPN/\cI$ is an isomorphism between analytic quotients. Each of the ideals listed in the statement of this theorem is, by Lemma~\ref{L.StronglyCountablyDetermined}, strongly countably determined. Therefore Theorem~\ref{T.OCA-l} implies that $\Phi$ has a continuous lifting on an ideal that intersects every perfect tree-like almost disjoint family. Such an ideal is by Theorem~\ref{T.HNM} (and Example~\ref{ex.perfect}) nonmeager.  Since $\Phi$ is an isomorphism, this implies that $\Phi$ has a continuous lifting. If $\cI$ is a nonpathnological  anaytic P-ideal, then it $\Phi$ has a completely additive lifting by \cite[Theorem~1.9.1]{Fa:AQ}. If $\cI$ is a nonpathological $F_\sigma$ ideal or one of the ideals $\NWD(\bbQ)$, $\NULL(\bbQ)$, or $\cZ_W$, then $\Phi$ has a completely additive lifting by \cite{KanRe:New}, \cite{KanRe:Ulam}. 
\end{proof}
\section{Concluding remarks}

A metric version of Biba's trick was used in \cite[Proposition~2.10]{de2024metric}. 

The most important question related to this work is whether lifting results from forcing axioms can be extended to analytic ideals of complexity higher than~$F_{\sigma\delta}$. The most embarrassing question along these lines is whether some forcing axiom---\emph{any} forcing axiom, including $\textrm{MM}^{+++}$ of \cite{Viale:Category}---implies that all automorphisms of $\cP(\bbN^2)/\Fin\times\Fin$ are trivial, where $\Fin\times\Fin$ is the ideal of all $A\subseteq \bbN^2$ such that  all but finitely many of the vertical sections of $A$ are finite. 

I do not know whether the assumptions of Theorem~\ref{T.OCA-l}  can be weakened to $\OCAT$.  The use of $\MAsigma$ in Proposition~\ref{P.uniformization} appears to be substantial. See the remarks preceding   Lemma~\ref{L.approximation.nonmeager}. This of course leaves a possibility that a completely different proof circumvents a need for $\MAsigma$. 
Theorem~\ref{T.OCAsharp-Fin} suffices for all topological applications of the OCA lifting theorem given in \cite[\S 4]{Fa:AQ}. Thus $\MA$ is not needed for any of these results. The conclusion of the analogous theorem proven from $\OCAT$ and $\MA$ is that $\Phi$ has a continuous lifting on an ideal that is ccc over $\Fin$, a property introduced in \cite{Fa:AQ} stronger than being nonmeager. Having this property is useful for some applications  to the rigidity of \v Cech--Stone remainders (for recent examples see  \cite{dow2024autohomeomorphisms} and \cite{dow2024non}). Closer to the spirit of the present paper are the results on \v Cech--Stone remainders in \cite{vignati2024weak}, proved using $\OCAT$ and $\MA$. These results rely on deep, difficult, and general results from \cite{vignati2022rigidity}, and I don't know whether they can be proven from $\OCAT$ alone.

By adding $\MAsigma$ to the assumptions and an additional argument that involves \cite[Lemma~2.3]{Ve:OCA} ($\MAsigma$ suffices for the conclusion of this lemma), it is possible to improve the conclusion of Theorem~\ref{T.OCAsharp-Fin} to 
`$\Phi$ has a continuous lifting on an ideal $\Jcont$ that is ccc over $\Fin$'.  
Contrary to what had been conjectured in~\cite{Fa:AQ}, no assumption short of $0=1$ suffices to improve the conclusion of Theorem~\ref{T.OCAsharp-Fin} to `$\Phi$ has an additive lifting'. The Stone duality reformulation of this assertion is more natural ($\beta\bbN$ is the \v Cech--Stone compactification of $\bbN$): There is a continuous map $f\colon \beta\bbN\setminus \bbN\to \beta\bbN\setminus \bbN$ that cannot be continuously extended to a map from $\beta\bbN$ into $\beta\bbN$. 
This was proved in \cite{dow2014non}.

\bibliography{ifmainbib}

\begin{thebibliography}{10}

\bibitem{de2023trivial}
B.~De~Bondt, I.~Farah, and A.~Vignati.
\newblock Trivial isomorphisms between reduced products.
\newblock {\em Israel J. Math.}, to appear.

\bibitem{de2024metric}
B.~De~Bondt and A.~Vignati.
\newblock A metric lifting theorem.
\newblock {\em arXiv preprint arXiv:2411.11127}, 2024.

\bibitem{dow2014non}
A.~Dow.
\newblock A non-trivial copy of {$\beta N \backslash N$}.
\newblock {\em Proc. Amer. Math. Soc.}, 142(8):2907--2913, 2014.

\bibitem{dow2024autohomeomorphisms}
A.~Dow.
\newblock Autohomeomorphisms of pre-images of {$\bbN^*$}.
\newblock {\em arXiv preprint arXiv:2406.09319}, 2024.

\bibitem{dow2024non}
A.~Dow.
\newblock Non-trivial copies of {$\bbN^*$}.
\newblock {\em arXiv preprint arXiv:2406.03471}, 2024.

\bibitem{DoHa:Lebesgue}
A.~Dow and K.P. Hart.
\newblock The measure algebra does not always embed.
\newblock {\em Fundamenta Mathematicae}, 163:163--176, 2000.

\bibitem{Fa:AQ}
I.~Farah.
\newblock {\em Analytic quotients: theory of liftings for quotients over
  analytic ideals on the integers}.
\newblock Number 702 in Memoirs AMS, vol. 148. 2000.

\bibitem{Fa:Luzin}
I.~Farah.
\newblock Luzin gaps.
\newblock {\em Trans. Amer. Math. Soc.}, 356:2197--2239, 2004.

\bibitem{farah2022corona}
I.~Farah, S.~Ghasemi, A.~Vaccaro, and A.~Vignati.
\newblock Corona rigidity.
\newblock {\em Bull. Symb. Logic}, to appear.

\bibitem{FoMaShe:Martin}
M.~Foreman, M.~Magidor, and S.~Shelah.
\newblock Martin's maximum, saturated ideals and nonregular ultrafilters, {I}.
\newblock {\em Ann. of Math. (2)}, 127:1--47, 1988.

\bibitem{Fr:Farah}
D.H. Fremlin.
\newblock Notes on {\sc {f}arah} {P}99.
\newblock preprint, University of Essex, available at {\tt
  http://www.essex.ac.uk/maths/staff/fremlin/preprints.htm}, June 1999.

\bibitem{he2022borel}
X.~He, H.~Zhang, and S.~Zhang.
\newblock The {B}orel complexity of ideal limit points.
\newblock {\em Top. Appl.}, 312:108061, 2022.

\bibitem{JN}
S.-A. Jalali-Naini.
\newblock {\em The monotone subsets of Cantor space, filters and descriptive
  set theory}.
\newblock PhD thesis, Oxford, 1976.

\bibitem{Just:WAT}
W.~Just.
\newblock A weak version of {AT} from {OCA}.
\newblock {\em MSRI Publications}, 26:281--291, 1992.

\bibitem{KanRe:Ulam}
V.~Kanovei and M.~Reeken.
\newblock On {U}lam's problem concerning the stability of approximate
  homomorphisms.
\newblock {\em Tr. Mat. Inst. Steklova}, 231:249--283, 2000.

\bibitem{KanRe:New}
V.~Kanovei and M.~Reeken.
\newblock New {R}adon--{N}ikodym ideals.
\newblock {\em Mathematika}, 47:219--227, 2002.

\bibitem{Ke:Classical}
A.S. Kechris.
\newblock {\em Classical descriptive set theory}, volume 156 of {\em Graduate
  Texts in Mathematics}.
\newblock Springer, 1995.

\bibitem{Ku:Box}
K.~Kunen.
\newblock Some comments on box products.
\newblock In A.~Hajnal et~al., editors, {\em Infinite and finite sets,
  Keszthely (Hungary), 1973}, volume~10 of {\em Coll.~Math.~Soc.~J\"anos
  Bolyai}, pages 1011--1016. North-Holland, Amsterdam, 1975.

\bibitem{Ku:Set}
K.~Kunen.
\newblock {\em Set theory}, volume~34 of {\em Studies in Logic (London)}.
\newblock College Publications, London, 2011.

\bibitem{Laf:Forcing}
C.~Laflamme.
\newblock Forcing with filters and complete combinatorics.
\newblock {\em Ann. Pure Appl. Logic}, 42:125--163, 1989.

\bibitem{magidor201some}
M.~Magidor.
\newblock Some set theories are more equal.
\newblock In P.~Koellner, editor, {\em Exploring the Frontiers of
  Incompleteness}. to appear.

\bibitem{Maz:F-sigma}
K.~Mazur.
\newblock {$F_\sigma$}-ideals and $\omega_1\omega_1^*$-gaps in the {B}oolean
  algebra {$\p(\omega)/I$}.
\newblock {\em Fundamenta Mathematicae}, 138:103--111, 1991.

\bibitem{moore2021some}
J.T. Moore.
\newblock Some remarks on the open coloring axiom.
\newblock {\em Ann. Pure Appl. Logic}, 172(5):102912, 2021.

\bibitem{Sh:Proper}
S.~Shelah.
\newblock {\em Proper Forcing}.
\newblock Lecture Notes in Mathematics 940. Springer, 1982.

\bibitem{ShSte:PFA}
S.~Shelah and J.~Stepr{\=a}ns.
\newblock {PFA} implies all automorphisms are trivial.
\newblock {\em Proceedings of the American Mathematical Society},
  104:1220--1225, 1988.

\bibitem{Sol:AnalyticII}
S.~Solecki.
\newblock Analytic ideals and their applications.
\newblock {\em Ann. Pure Appl. Logic}, 99:51--72, 1999.

\bibitem{Ta:Compacts}
M.~Talagrand.
\newblock Compacts de fonctions mesurables et filters nonmesurables.
\newblock {\em Studia Mathematica}, 67:13--43, 1980.

\bibitem{Tal:Maharam}
M.~Talagrand.
\newblock Maharam's problem.
\newblock {\em Annals of Math.}, 168(3):981--1009, 2008.

\bibitem{To:Partition}
S.~Todorcevic.
\newblock {\em Partition Problems in Topology}, volume~84 of {\em Contemporary
  mathematics}.
\newblock American Mathematical Society, Providence, Rhode Island, 1989.

\bibitem{Ve:Definable}
B.~Velickovic.
\newblock Definable automorphisms of {$\cP(\omega) /\Fin$}.
\newblock {\em Proceedings of the American Mathematical Society}, 96:130--135,
  1986.

\bibitem{Ve:OCA}
B.~Veli{\v{c}}kovi{\'c}.
\newblock {OCA} and automorphisms of {${\mathcal P}(\omega) /\Fin$}.
\newblock {\em Top. Appl.}, 49:1--13, 1993.

\bibitem{Viale:Category}
M.~Viale.
\newblock Category forcings, {MM}$^{+++}$, and generic absoluteness for the
  theory of strong forcing axioms.
\newblock {\em J. Amer. Math. Soc.}, 29(3):675--728, 2016.

\bibitem{vignati2022rigidity}
A.~Vignati.
\newblock Rigidity conjectures for continuous quotients.
\newblock In {\em Ann. Sci. Ec. Norm. Super.}, volume~55, pages 1687--1738,
  2022.

\bibitem{vignati2024weak}
A.~Vignati and D.~Yilmaz.
\newblock The weak extension principle.
\newblock {\em arXiv preprint arXiv:2407.20791}, 2024.

\end{thebibliography}
\bibliographystyle{plain}

\end{document}